\documentclass[a4paper]{amsart}

\usepackage[utf8]{inputenc}
\usepackage[british]{babel}
\usepackage{amsfonts,amsmath,amssymb,amsthm,mathtools,esint}
\usepackage{tikz}
\usepackage{enumitem}
\usepackage[hidelinks]{hyperref}
\usepackage{cleveref}
\crefname{subsection}{subsection}{subsections}
\usepackage{thmtools}
\numberwithin{equation}{section}
\usepackage{xcolor}
\definecolor{dark-gray}{gray}{0.3}
\usepackage{caption}
\newcommand{\R}{\mathbb{R}}

\renewcommand{\d}[1]{\,\mathrm{d}#1}
\renewcommand{\div}{\mathrm{div}}
\newcommand{\I}{\mathrm{I}}
\newcommand{\vertiii}[1]
{{\left\vert\kern-0.25ex\left\vert\kern-0.25ex\left\vert #1 
\right\vert\kern-0.25ex\right\vert\kern-0.25ex\right\vert}}

\newcommand{\Fcal}{\mathcal{F}}

\newcommand{\Mcal}{\mathcal{M}}

\newcommand{\RT}{\mathrm{RT}}
\newcommand{\pw}{\mathrm{pw}}
\newcommand{\D}{\mathrm{D}}
\newcommand{\s}{\mathrm{s}}

\newtheorem{theorem}{Theorem}[section]
\newtheorem{lemma}[theorem]{Lemma}

\newtheorem{corollary}[theorem]{Corollary}
\newtheorem{proposition}[theorem]{Proposition}
\theoremstyle{remark}
\newtheorem{remark}[theorem]{Remark}

\newcounter{cntS}

\newcounter{cntL}

\usepackage[backend=bibtex, giveninits=true, isbn=false, doi=false, url=false, maxnames = 6]{biblatex}
\DeclareFieldFormat[article]{title}{{#1}} 
\DeclareFieldFormat[book]{title}{{#1}}
\DeclareFieldFormat[inproceedings]{title}{{#1}}
\DeclareFieldFormat[incollection]{title}{{#1}}
\DeclareFieldFormat{pages}{{#1}} 
\AtBeginBibliography{\footnotesize}
\begin{document}
\title[LDG for convex minimization]{Local discontinuous Galerkin \\ FEM for convex minimization}
\author[C.~Carstensen, N.~T.~Tran]{Carsten Carstensen \& Ngoc Tien Tran}
\thanks{The second author received funding from the European Union's Horizon 2020 
research and innovation programme (project RandomMultiScales, grant agreement No.~865751).}
\address[C.~Carstensen]{%
	Humboldt-Uni\-ver\-si\-t\"at zu Berlin,
	10117 Berlin, Germany}
\email{cc@math.hu-berlin.de}
\address[N.~T.~Tran]{Universit\"at Augsburg, 86159 Augsburg, Germany}
\email{ngoc1.tran@uni-a.de}
\begin{abstract}
The heart of the a priori and a posteriori error control in convex minimization problems 
    is the sharp control of the differences of discrete and exact minimal 
    energy. Conforming finite element discretizations  for p-Laplace type minimization problems
    provide upper bounds of the energy difference with optimal convergence rates.
    Even for smooth solutions, known convergence rates for higher-order non-conforming 
    finite element discretizations for the same problem class with $2 < p < \infty$,  however,  are exclusively 
    suboptimal. Thus the popular a posteriori 
    error control within the two-energy principle, that generalizes hyper-circle identities, 
    appears unbalanced. 
     
The innovative point of departure in a refined analysis of two discontinuous Galerkin 
    (dG) schemes exploits duality relations between a discrete 
    primal and a semi-discrete dual problem. The infinite-dimensional dual problem
    leads to a tiny duality gap that even vanishes for polynomial low-order terms.
    For a class of degenerated convex minimization problems with two-sided $p$ growth,
    the novel duality
    provides improved a priori convergence rates for the error in the minimal energies. 
    This closes the misfit of convergence rates for the conforming and nonconforming 
    schemes at least for the local discontinuous Galerkin schemes at hand.
    The  motivating two-energy principle and some post-processing for a Raviart-Thomas
    dual variable provides an a posteriori error control, that also 
    may drive adaptive mesh-refining. Computational benchmarks  provide   
    numerical evidence for improved convergence rates of the adaptive beyond uniform
    mesh-refining.
\end{abstract}

\maketitle

\section{Introduction}\label{sec:main-results}
This paper develops novel techniques to establish duality relations for higher-order nonconforming methods for convex minimization problems with application to the error analysis of local discontinuous Galerkin methods (LDG).

\subsection{Model problem}\label{sec:continuous-problem}
Given an open bounded polyhedral Lipschitz domain $\Omega \subset \mathbb{R}^n$, the continuous problem minimizes the energy
\begin{align}
	E(v) \coloneqq \int_\Omega (W(\nabla v) + \psi(\bullet,v)) \d{x} \quad\text{among } v \in V \coloneqq W^{1,p}_\mathrm{D}(\Omega)\label{def:energy}
\end{align}
with  a convex energy density $W \in C(\mathbb{R}^n)$ and with 
a measurable function $\psi : \Omega \times \mathbb{R} \to \mathbb{R} \cup \{+\infty\}$,  convex
in the second variable: At a.e. $x \in \Omega$ let 
$\psi(x, \bullet):\R\to\R $ be a  proper lower semi-continuous convex function.
Underlying non-displayed growth conditions lead to 
$1<p<\infty$  and the test space 
$V = W^{1,p}_\mathrm{D}(\Omega)$ of Sobolev functions in $W^{1,p}(\Omega)$ with homogeneous boundary data on a compact part $\Gamma_\mathrm{D} \subset \partial \Omega$ of the boundary $\partial \Omega$ with 
positive surface measure. 

The convexity in the functions inside \eqref{def:energy} gives rise to a dual energy: 
Let $W^*$ (resp.~$\psi^*(x,\bullet)$) denote the convex conjugate of $W$ (resp.~$\psi(x,\bullet)$ for a.e.~$x \in \Omega$) and 
the H\"older conjugate 
$1< p' \coloneqq p/(p-1) <\infty$ of $p$, $1/p + 1/p' = 1$.
The dual problem of \eqref{def:energy} maximizes the dual energy
\begin{align}\label{def:dual-energy}
	E^*(\tau) \coloneqq -\int_\Omega (W^*(\tau) + \psi^*(\bullet, \div \tau)) \d{x}
    \quad\text{among } \tau \in \Sigma \coloneqq W^{p'}_\mathrm{N}(\div,\Omega). 
\end{align}
The space $\Sigma$ consists 
of vector fields 
$\tau \in L^{p'}(\Omega)^n$ with a distributional divergence $\div\, \tau$ in  $L^{p'}(\Omega)$ 
and normal traces $\tau \cdot \nu=0$, that vanish along the 
(relatively open and possibly empty) Neumann boundary $\Gamma_\mathrm{N} \coloneqq \partial \Omega \setminus \Gamma_\mathrm{D}$. It is a Banach space
under the graph norm as, e.g., in the Hilbert case  $W^2_\mathrm{N}(\div,\Omega) = H_\mathrm{N}(\div,\Omega)$.
Throughout this paper, the subsequent general conditions 
\begin{enumerate}[wide]
    \item[(A1)] $-\infty < \inf E(V) \leq E(v_0) < \infty$ for some $v_0 \in V$ and $E$ is continuous at $v_0$,
    \item[(A2)] $E(v) \to \infty$ as $\|\nabla v\|_p \to \infty$ for $v \in V$
\end{enumerate}
first  imply existence of solutions  $u \in \arg\min E(V)$ 
and $\sigma \in \arg \max E^*(\Sigma)$, as well as, second,  no duality gap 
\cite[Chapter 3, Theorem 4.2]{EkelandTeman1999} viz.
\begin{align}\label{eq:strong-duality}
    E(u) = \min E(V) = \max E^*(\Sigma) = E^*(\sigma).
\end{align}
 
\subsection{Motivation}
The a~priori error analysis of conforming finite element methods (FEM)
is well understood in the literature \cite{GlowinskiMarrocco1975,Chow1989,CPlechac1997}. 
We illustrate the main results on the a priori convergence rates from a class
of problems with $p \geq 2$ and generic constants $c_1,c_3,c_5 > 0$ and $c_2,c_4 \geq 0$ in the subsequent assumptions. 
\begin{enumerate}[wide]
    \item[(B1)] (smoothness of $W$) $W \in C^1(\mathbb{R}^n)$.
    \item[(B2)] (two-sided growth) Any $a \in \mathbb{R}^n$ satisfies
$       c_1|a|^p - c_2 \leq W(a) \leq c_3|a|^p + c_4  $.
    \item[(B3)] (convexity control)
Any $a,b \in \mathbb{R}^n$ satisfy\begin{align*}
    d(a,b) \coloneqq \frac{ |\D W(a) - \D W(b)|^2}{ 1 + |a|^{p-2} + |b|^{p-2}}
\leq c_5\left(W(b) - W(a) - \D W(a) \cdot (b - a)\right).
\end{align*}   
\item[(B4)] (linear low-order term) $\psi(x,a) \coloneqq -f(x) a$ 
for some $f \in L^{p'}(\Omega)$.
\end{enumerate}
The assumptions (B1)--(B3) define a class of degenerate convex minimization problems:
The dual variable $\sigma = \D W(\nabla u)$ is uniquely defined, i.e.,  
independent of the choice of the possibly non-unique minimizer $u$, 
and $\sigma$ belongs to  $H^1_{\rm loc}(\Omega)^n$ 
\cite{CPlechac1997,CarstensenMueller2002}.
Examples include an optimal design problem \cite{KohnStrang1986,BartelsC2008} and relaxed 
scalar double well problems \cite{CPlechac1997}.
Although the explicit error control of this paper considers the linear lower term $\psi$ from (B4), our arguments allow more general $\psi$ such as the $L^2$ penalty $\psi(x,a) = - f(x)a + |a - g(x)|^2$ from \cite{CPlechac1997}.

Given a discrete minimizer $u_h \in \arg\min E(V_h)$ of the energy 
\eqref{def:energy}
in a conforming subspace  $V_h \subset V$,  let $\sigma_h \coloneqq \D W(\nabla u_h)$ denote the discrete dual variable. 
Under the assumptions (B1)--(B4),
arguments from \cite{GlowinskiMarrocco1975,Chow1989,CPlechac1997} imply 
\begin{align}\label{confenergyerror}
E(u_h) - E(u) \lesssim \inf_{v_h \in V_h} \|\nabla(u - v_h)\|_p^2 
=O(h_{\max}^{2k})
\end{align}
with the maximal mesh-size $h_{\max}$ of the underlying finite element mesh
and under sufficient smoothness assumptions 
with the discretization order $k \geq 1$. For the convenience of the reader, we provide a proof of \eqref{confenergyerror} in the appendix.
This control of the error in the energies \eqref{confenergyerror}  enables several
error estimates for the error in the primal and dual variables
in \cite{GlowinskiMarrocco1975,Chow1989,CPlechac1997}.

The a priori error analysis of higher order nonconforming schemes is less developed
and we cannot even quote directly an analog of \eqref{confenergyerror} from the literature. Before
we present  \eqref{nonconfenergyerror},
we point out  that known  a priori error estimates for higher-order nonconforming schemes are  
really suboptimal:  For strongly monotone problems and $p\ge 2$,
\cite{DiPietroDroniou2017-II,DroniouEymardGallouet2018} 
(and the references therein) solely 
establish 
$\|\nabla_\pw(u-u_h)\|_{L^p(\Omega)} =O(h_{\max}^{k/(p-1)})$,
which is inferior to \eqref{confenergyerror}
for the
conforming FEM with \eqref{confenergyerror} \cite{Chow1989}.

A similar sub-optimality  arises in the case $1 < p < 2$, which has been resolved only recently in \cite{Tran2024} with weak duality 
between HHO and discrete dual hybrid methods.
In the current case $p\ge 2$, however, those
arguments cannot fully mimic the proof of 
\eqref{confenergyerror} and led  to suboptimal rates.

The main motivation to apply nonconforming methods comes from a direct 
a~posteriori error control by duality 
\cite{Repin1997,Repin2000,NeittaanmakiRepin2004,CLiu2015,Bartels2015} in 
degenerate problems, 
when residual-based error estimation suffer from the reliability-efficiency gap \cite{CJochimsen2003}.

The smoothness assumptions for  \eqref{confenergyerror}-\eqref{nonconfenergyerror} may appear unrealistically high,
but behind those short statements is an error analysis that reduces the convergence of the schemes to that of
interpolation errors. On the theoretical side, optimal rates indicate a sharp analysis. On the practical side
we might expect that an adaptive mesh-refining miraculously  removes singularity-driven suboptimal global 
approximation errors: Computational benchmarks may confirm this wishful vision.

\subsection{Contributions of this paper}
This paper establishes the analog of \eqref{confenergyerror} for 
local discontinuous Galerkin (LDG) methods 
\cite{CockburnShu1998,BurmanErn2008,DiPietroErn2012}, viz.
\begin{align}\label{nonconfenergyerror}
E_h(\mathrm{I}_h u) - \min E_h(V_h) + |E(u) - \min E_h(V_h)|
=O(h_{\max}^{2k})
\end{align}
with novel strong duality relations on the discrete level by employing continuous objects in the dual ansatz space.
This allows an improved a~priori error estimate
\begin{align}\label{ineq:energy_diff_err}
	0 \leq E_h(\mathrm{I}_h u) - \min E_h(V_h) \leq E_h(\mathrm{I}_h u) - E(u) + \gamma_h(\mathrm{I}_h^* \sigma)/r'.
\end{align}
Here and in \eqref{nonconfenergyerror}, 
$E_h$ denotes a discrete version of \eqref{def:energy}, 
$V_h$ is the discrete ansatz space, 
and $\mathrm{I}_h$ and $\mathrm{I}_h^*$ are interpolation operators.
In contrast to \cite{Tran2024}, the discrete dual energy $E^*_h(\mathrm{I}_h^* \sigma) = E^*(\sigma) + \gamma_h(I_h^* \sigma)/r'$ of the interpolation $\mathrm{I}_h^* \sigma$ can be explicitly expressed in terms of the exact energy $E^*(\sigma) = E(u)$ and the stabilization $\gamma_h(\mathrm{I}_h^* \sigma)/r'$ on the discrete dual level thanks to the larger ansatz space with continuous functions. Therefore, the upper bound in \eqref{ineq:energy_diff_err} avoids, up to stabilization, the unbalanced energy difference $E^*(\sigma) - E^*_h(\mathrm{I}_h^* \sigma)$.
Under the assumptions (B1)-(B4) and
piecewise 
smoothness assumptions $u \in W^{k+1,p}$ and $\sigma \in W^{k,2}$, we provide squared convergence rates in \eqref{nonconfenergyerror}.
This enables 
error estimates in the primal and dual variables as in \cite[Section 4]{Tran2024}, e.g.,
$\|\nabla_\pw(u-u_h)\|_{p} = O(h_{\max}^{2k/p})$ follows
for $p \geq 2$ in strongly monotone problems (in comparison to the older result $O(h_{\max}^{k/(p-1)})$ in the literature \cite{DiPietroDroniou2017-II,DroniouEymardGallouet2018,Tran2024}). 
This closes the theoretical gap between conforming and nonconforming discretizations.

Our results carry over to hybridizable methods with  HDG/WG stabilizations \cite{CockburnGopalakrishnanLazarov2009} with suboptimal polynomial consistency ~\cite[Remark 2.9]{DiPietroDroniou2017}; however, it does not apply 
to the Lehrenfeld-Sch\"oberl and HHO stabilization.

A conforming dual  Raviart-Thomas finite element functions  approximation 
enables guaranteed energy error bounds. We suggest a post-processing by
explicit design of the required degrees of freedom as in \cite{ErnStephansenVohralik2010} for linear problems. 
A localization of the resulting error estimator drives an 
adaptive mesh-refining algorithms as an alternative to \cite{CarstensenTran2021}. 

\subsection{Outline}
The remaining parts of this paper are organized as follows. \Cref{sec:discretization} introduces the numerical methods considered in this paper. The equivalence of these methods to dual maximization problems is established in \Cref{sec:equivalence}. This applies to the error analysis of an LDG method in \Cref{sec:error-analysis}. 
An extension of the analysis to hybridizable methods is briefly discussed in \Cref{sec:hybrid}.
Three numerical benchmarks in \Cref{sec:numerical_examples} 
with improved convergence rates for adaptive mesh-refining algorithms
conclude this paper.  

\subsection{Notation}
Standard notation for Sobolev and Lebesgue spaces applies throughout this paper with the abbreviation $\|\bullet\|_p \coloneqq \|\bullet\|_{L^p(\Omega)}$ for any $1 < p < \infty$.
The notation $A \lesssim B$ abbreviates $A \leq CB$ for a generic constant $C$ independent of the mesh-size and $A \approx B$ abbreviates $A \lesssim B \lesssim A$.

\section{Discretization}\label{sec:discretization}
This section presents the numerical scheme for the discretization of \eqref{def:energy}.

\subsection{Polytopal Mesh}\label{sec:triangulation}
Let $\mathcal{M}$ be a finite collection of closed polytopes of positive volume with overlap of measure zero that covers $\overline{\Omega} = \cup_{K \in \mathcal{M}} K$.
A face $S$ of the mesh $\mathcal{M}$ is a closed connected subset of a hyperplane $H_S$ with positive $(n-1)$-dimensional surface measure such that either (a) there exist $K_+,K_- \in \mathcal{M}$ with $S \subset H_S \cap K_+ \cap K_-$ (interior face) or (b) there exists $K_+ \in \mathcal{M}$ with $S \subset H_S \cap K_+ \cap \partial \Omega$ (boundary face). We refer to \cite[Section 1.1]{DiPietroDroniou2017} and \cite[Sections 1.2.2--1.2.3]{DiPietroErn2012} for further details.

Let $\mathcal{F}$ be a finite collection of faces with overlap of $(n-1)$-dimensional surface measure zero that covers the skeleton $\partial \Mcal \coloneqq \cup_{K \in \mathcal{M}} \partial K = \cup_{S \in \mathcal{F}} S$ with the split $\mathcal{F} = \mathcal{F}(\Omega) \cup \mathcal{F}(\partial \Omega)$ into the set of interior faces $\mathcal{F}(\Omega)$ and the set of boundary faces $\mathcal{F}(\partial \Omega)$.
Let $\mathcal{F}_\mathrm{D} \coloneqq \{S \in \mathcal{F} : S \subset \Gamma_\mathrm{D}\}$ (resp.~$\mathcal{F}_\mathrm{N} \coloneqq \mathcal{F}(\partial\Omega) \setminus \mathcal{F}_\mathrm{D}$) denote the set of Dirichlet (resp.~Neumann) faces. For $K \in \mathcal{M}$, $\mathcal{F}(K)$ is the set of all faces of $K$.
The normal vector $\nu_S$ of an interior face $S \in \mathcal{F}(\Omega)$ is fixed in its orientation beforehand and set $\nu_S \coloneqq \nu|_S$ for boundary faces $S \in \mathcal{F}(\partial \Omega)$.
For $S \in \mathcal{F}(\Omega)$, $K_+ \in \Mcal_h$ (resp.~$K_- \in \Mcal$) denotes the unique cell with $S \subset \partial K_{+}$ (resp.~$S \subset \partial K_-$) and $\nu_{K_+}|_S = \nu_S$ (resp.~$\nu_{K_-}|_S = -\nu_S$).
For $S \in \mathcal{F}(\partial \Omega)$, $K_+ \in \Mcal$ is the unique cell with $S \subset \partial K_+$.
The jump $[v]_S$ and the average $\{v\}_S$ of any function $v \in W^{1,1}(\mathrm{int}(K_+ \cup K_-))$ along $S \in \mathcal{F}(\Omega)$ are defined by $[v]_S \coloneqq v|_{K_+} - v|_{K_-} \in L^1(S)$ and $\{v\}_S \coloneqq (v|_{K_+} + v|_{K_-})/2 \in L^1(S)$.
If $S \in \mathcal{F}(\partial \Omega)$, then $[v]_S \coloneqq v|_S = v_{K_+}|_S \eqqcolon \{v\}_S$.

For theoretical purposes, let $\vartheta$ denote the mesh regularity parameter of $\Mcal$ associated with a matching simplicial submesh, we refer to \cite[Definition 1.38]{DiPietroErn2012} for a detailed definition. The constants in discrete inequalities such as the trace or inverse inequality depend on this parameter.
The differential operators $\nabla_\pw$ and $\div_\pw$ denote the piecewise version of $\nabla$ and $\div$ without explicit reference to the underlying mesh.

\subsection{Finite element spaces}\label{sec:fem-spaces}
Given a subset $M \subset \R^n$ of diameter $h_M$, let $P_k(M)$ denote the space of polynomials of degree at most $k$.
For any $v \in L^1(M)$, $\Pi_M^k v \in P_k(M)$ denotes the $L^2$ projection of $v$ onto $P_k(M)$;
$P_k(\mathcal{M})$ and $P_k(\mathcal{F})$ denote 
the space of piecewise polynomials of degree at most $k$ with respect to the mesh $\mathcal{M}$ and the faces $\mathcal{F}$.
The piecewise constant function $h_\Mcal \in P_0(\Mcal)$ reads $h_\Mcal|_K = h_K = \mathrm{diam}(K)$; $h_{\max} \coloneqq \max_{K \in \Mcal} h_K$ is the maximal mesh-size of $\Mcal$.

\subsection{Modified local discontinuous Galerkin method}\label{sec:discrete-spaces}
For any $k \geq 1$,
we consider the discrete ansatz space 
$$V_h \coloneqq P_{k}(\mathcal{M}).$$
The discrete gradient $\nabla_h v_h \in P_{k-1}(\Mcal)^n$ of $v_h \in V_h$ is the unique solution to
\begin{align}\label{def:discrete-gradient}
 	\int_\Omega \nabla_h v_h \cdot \Phi \d{x} = - \int_\Omega v_h\, \div_\pw \Phi \d{x} + \sum_{S \in \mathcal{F}\setminus\mathcal{F}_\mathrm{D}} \int_S \{v_h\}_S [\Phi]_S \cdot \nu_S \d{s}
\end{align}
for any $\Phi \in P_{k-1}(\Mcal)^n$.
An integration by parts proves
\begin{align}\label{def:ldg-dicrete-gradient}
	\int_\Omega \nabla_h v_h \cdot \Phi \d{x}
	= \int_\Omega \nabla_\pw v_h \cdot \Phi \d{x} - \sum_{S \in \mathcal{F}\setminus\mathcal{F}_\mathrm{N}} \int_S [v_h]_S \{\Phi\}_S\cdot \nu_S \d{s}.
\end{align}
Note that $\nabla_h$ coincides with the discrete gradient of \cite[Section 4.3.2]{DiPietroErn2012} and can be expressed in terms of $\nabla_\pw$ and lifting operators \cite{BrezziManziniMariniPietraRusso2000}. The error between $\nabla_\pw v_h$ and $\nabla_h v_h$ is controlled by the interior penalty stabilization.
\begin{lemma}[discrete consistency error]\label{lem:discrete-consistency-error}
	Any $v_h \in V_h$ satisfies
	\begin{align*}
		\|\nabla_\pw v_h - \nabla_h v_h\|_{p}^p \lesssim \sum_{S \in \mathcal{F} \setminus \mathcal{F}_\mathrm{N}} h_S^{1-p}\|[v_h]_S\|_{L^p(S)}^p.
	\end{align*}
\end{lemma}
\begin{proof}
	Given any $v_h \in V_h$ and $\Phi_h \in P_{k-1}(\Mcal)^n$, a discrete trace inequality in \eqref{def:ldg-dicrete-gradient} proves
	\begin{align*}
		\int_\Omega (\nabla_\pw v_h - \nabla_h v_h) \cdot \Phi_h \d{x} \lesssim \Big(\sum_{S \in \mathcal{F} \setminus \mathcal{F}_\mathrm{N}} h_S^{1-p}\|[v_h]_S\|_{L^p(S)}^p\Big)^{1/p}\|\Phi_h\|_{p'}.
	\end{align*}
	Since $\|\Pi_\Mcal^k \Phi\|_{p'} \lesssim \|\Phi\|_{p'}$ for any $\Phi \in L^{p'}(\Omega)^n$ from stability of the $L^2$ projection in $L^{p'}(\Omega)^n$ \cite[Lemma 3.2]{DiPietroDroniou2017}, the proof concludes with
	\begin{align*}
		\|\nabla_\pw v_h - \nabla_h v_h\|_{p} &= \sup_{\Phi \in L^{p'}(\Omega)^n \setminus \{0\}} \int_\Omega (\nabla_\pw v_h - \nabla_h v_h) \cdot \Pi_\Mcal^{k-1} \Phi \d{x}/\|\Phi\|_{p'}.\qedhere
	\end{align*}
\end{proof}
Let $\psi_h : \Omega \times \mathbb{R}^m \to \mathbb{R} \cup \{+\infty\}$ be an approximation of $\psi$ (so that $\psi_h(x,\bullet)$ is a proper lower semicontinuous convex function for a.e.~$x \in \Omega$).
Given fixed parameters $1<r<\infty$ and $s \in \mathbb{R}$,
the LDG method of this paper minimizes the discrete energy
\begin{align}\label{def:discrete_energy}
	E_h(v_h) &\coloneqq \int_\Omega (W(\nabla_h v_h) 
    + \psi_h(\bullet,v_h)) \d{x} + s_h(v_h)/r,\\
    s_h(v_h; w_h) &\coloneqq 
    \sum_{S \in \mathcal{F} \setminus \mathcal{F}_\mathrm{N}} 
    h_S^{-s}\int_S |[v_h]_S|^{r-2} [v_h]_S\,[w_h]_S \d{s}
\label{def:stabilization-ldg}
\end{align}
for any $v_h, w_h \in V_h$
and the convention $s_h(v_h) \coloneqq s_h(v_h;v_h)$.
In the linear case with $W(a) \coloneqq |a|^2/2$ (and $r =2$, $s = 1$), 
this leads to a local discontinuous Galerkin method \cite{CockburnShu1998}, 
cf.~also \cite[Section 4.4.2]{DiPietroErn2012}. For the $p$-Laplace problem, 
this method was proposed in \cite{BurmanErn2008} with $r = p$ and $s = p-1$.
For the existence of discrete minimizers, we assume 
corresponding discrete versions of (A1)-(A2).
\begin{remark}[other DG methods]
The design of DG methods for \eqref{def:energy} is delicate because the convex energy structure may be forfeited \cite{OrtnerSueli2007,GrekasKoumatosMakridakisVikelis2025}.
The DG methods in \cite{EyckLew2006,BurmanErn2008,BuffaOrtner2009} utilized a reconstruction operator for the discretization of the continuous gradient, which preserves the convexity structure of the energy on the discrete level.
For the lowest-order discretization on regular triangulations into simplices, the DG methods of \cite{Bartels2021} provide a simple approach by approximating the continuous gradient with the piecewise one, but the analysis requires properties 
of Crouzeix-Raviart and Raviart-Thomas finite element functions.
\end{remark}


\section{Duality relations on discrete level}\label{sec:equivalence}
The main tools for the analysis of this paper are duality relations of the discrete problem \eqref{def:discrete_energy} to a dual maximization problem with the ansatz space
\begin{align*}
	Y &\coloneqq \{(\tau_\Mcal, \tau_\mathcal{F}) \in L^{p'}(\Omega)^n \times L^{r'}(\partial \mathcal{M}) : \tau_\mathcal{F}|_S \equiv 0 \text{ for all } S \in \mathcal{F}_\mathrm{N}\}.
\end{align*}
Given any function $\tau \in W^{1,t}(\Omega)^n \cap \Sigma$ with $t = \frac{nr'}{r'+n-1} > 1$,
we define the interpolation
\begin{align*}
	\mathrm{I}_h^* \tau \coloneqq (\tau, (\tau \cdot \nu_S)_{S \in \mathcal{F}}) \in Y.
\end{align*}
Note that, by Sobolev embeddings \cite[Theorem 4.12]{AdamsFournier2003}, the assumption $\tau \in W^{1,t}(\Omega)^n$ is sufficient to guarantee the Lebesgue integrability of traces of $\tau$ along $n-1$ dimensional hyperplanes, but this condition can be relaxed utilizing the additional smoothness of $\Sigma$, cf.~\Cref{rem:global_reg_normal_trace} for further details.
The divergence reconstruction $\div_h \tau \in P_{k}(\Mcal)$ of $\tau = (\tau_\Mcal, \tau_\mathcal{F}) \in Y$ is the unique solution to
\begin{align}\label{def:div-rec}
	\int_\Omega \div_h \tau\, \phi \d{x} = -\int_\Omega \tau_\Mcal \cdot \nabla_\pw \phi \d{x} + \sum_{S \in \mathcal{F} \setminus \mathcal{F}_\mathrm{N}} \int_S \tau_S [\phi]_S \d{s} 
\end{align}
for any $\phi \in P_{k}(\Mcal)$.
The operator $\div_h$ is consistent in the following sense.
\begin{lemma}[consistency]\label{lem:consistency}
	Any $\tau \in W^{1,t}(\Omega)^n \cap \Sigma$ satisfies
		$\div_h \I_h^* \tau = \Pi_\Mcal^k \div\, \tau$.
\end{lemma}

\begin{proof}
	The right-hand side of \eqref{def:div-rec} is equal to $(\div\, \tau, \phi)_{L^2(\Omega)}$ for any $\phi \in P_{k}(\Mcal)$, which concludes the assertion.
\end{proof}
For any $\tau = (\tau_\Mcal, \tau_\mathcal{F}) \in Y$, consider the following dual energy
\begin{align}\label{def:dual-discrete-energy}
	E^*_h(\tau) &\coloneqq - \int_\Omega (W^*(\tau_\Mcal) + \psi_h^*(\bullet,\div_h \tau)) \d{x} - \gamma_h(\tau)/r',\\
    \gamma_h(\tau) &\coloneqq \sum_{S \in \mathcal{F}\setminus\mathcal{F}_\mathrm{N}} h^{s/(r-1)} \|\tau_S - \{\Pi_\Mcal^{k-1} \tau_\Mcal\}_S \cdot \nu_S\|_{L^{r'}(S)}^{r'}.\label{def:dual-stab}
\end{align}
The following duality relation between the LDG method \eqref{def:discrete_energy} and the dual problem of \eqref{def:dual-discrete-energy} holds. For strong duality, we assume the following condition.
\begin{enumerate}[wide]
    \item[(B5)] (polynomial low-order term)
    $\psi_h(x,\bullet) \in C^1(\mathbb{R})$ at a.e. $x\in \Omega$ and $\partial_u \psi_h(\bullet,v_h) \in P_k(\Mcal)$ for any $v_h \in V_h = P_k(\Mcal)$.
\end{enumerate}

\begin{theorem}[duality of LDG]\label{thm:weak-duality}
	It holds $\sup E_h^*(Y) \leq \min E_h(V_h)$; (B1) and (B5) imply 
		$\max E_h^*(Y) = \min E_h(V_h)$.
\end{theorem}

\begin{proof}
	Given $\tau = (\tau_\Mcal, \tau_\mathcal{F}) \in Y$ and $v_h \in V_h$,
    \eqref{def:ldg-dicrete-gradient} implies
    \begin{align*}
        &\int_\Omega \Pi_\Mcal^{k-1} \tau_\Mcal \cdot \nabla_h v_h \d{x}\\
        &\qquad= \int_\Omega \nabla_\pw v_h \cdot \Pi_\Mcal^{k-1} \tau_\Mcal \d{x} - \sum_{S \in \mathcal{F}\setminus\mathcal{F}_\mathrm{N}} \int_S [v_h]_S \{\Pi_\Mcal^{k-1} \tau_\Mcal\}_S \cdot \nu_S \d{s}.
    \end{align*}
    Since $\Pi_\Mcal^{k-1}$ can be omitted in the integral $(\nabla_\pw v_h, \Pi_\Mcal^{k-1} \tau_\Mcal)_{L^2(\Omega)}$,
    this and the definition of the divergence reconstruction in \eqref{def:div-rec} provide
    \begin{align}\label{eq:diby-hho}
        \int_\Omega \tau_\Mcal \cdot \nabla_h v_h \d{x} &= \int_\Omega \Pi_\Mcal^{k-1} \tau_\Mcal \cdot \nabla_h v_h \d{x} = -\int_\Omega \div_h \tau \,v_h \d{x}\nonumber\\
		&\qquad+ \sum_{S \in \mathcal{F} \setminus \mathcal{F}_\mathrm{N}} \int_S [v_h]_S (\tau_S - \{\Pi_\Mcal^{k-1} \tau_\Mcal\}_S \cdot \nu_S) \d{s}.
    \end{align}
	From this, $\tau_\Mcal \cdot \nabla_h v_h \leq W(\nabla_h v_h) + W^*(\tau_\Mcal)$ and $\div_h \tau_h \, v_h \leq \psi_h(\bullet,v_h) + \psi_h^*(\bullet,\div_h \tau)$ a.e.~in $\Omega$ as well as the H\"older inequality, we deduce $\sup E_h^*(Y) \leq \min E_h(V_h)$.
    To establish equality,
	let $u_h \in \arg\min E_h(V_h)$ be a minimizer of $E_h$ in $V_h$. We define the dual variable $y = (\sigma_\Mcal, \sigma_\mathcal{F}) \in Y$ with
	\begin{align}\label{def:dual-variable}
			\sigma_\Mcal \coloneqq \D W(\nabla_h u_h) \quad\text{and}\quad
			\sigma_\mathcal{F}|_S \coloneqq \{\Pi_{\Mcal}^{k-1} \sigma_\Mcal\}_S \cdot \nu_S - h_S^{-s}|[u_h]_S|^{r-2} [u_h]_S
	\end{align}
	for any $S \in \mathcal{F}\setminus\mathcal{F}_\mathrm{N}$.
	The Euler-Lagrange equations read
	\begin{align}\label{eq:dELE}
		0 = \int_\Omega (\sigma_\Mcal \cdot \nabla_h v_h \d{x} + \partial_u \psi_h(\bullet, u_h) v_h) \d{x} + s_h(u_h;v_h)
	\end{align}
	for any $v_h \in V_h$.
	This and
	the discrete
	integration by parts formula \eqref{eq:diby-hho} imply
	\begin{align}\label{eq:proof-mldg-equivalence}
		0 = &\int_\Omega (\partial_u \psi_h(\bullet, u_h) - \div_h y) v_h \d{x}\nonumber\\
		&\qquad + \sum_{S \in \mathcal{F} \setminus \mathcal{F}_\mathrm{N}} \int_S [v_h]_S \cdot (\sigma_S - \{\Pi_\Mcal^{k-1} \sigma_\Mcal\}_S \cdot \nu_S) \d{s} + s_h(u_h;v_h).
	\end{align}
	An explicit calculation with
	the definitions of $\sigma_\mathcal{F}$ from \eqref{def:dual-variable} and the stabilization $s_h$ from \eqref{def:stabilization-ldg} proves that the final two terms on the right-hand side cancel. Therefore, $0 = \int_\Omega (\partial_u \psi_h(\bullet, u_h) - \div_h y) v_h \d{x}$ holds for any $v_h \in V_h$ from \eqref{eq:proof-mldg-equivalence}. This yields
	\begin{align}\label{eq:div_h_sigma_h}
		\div_h y = \Pi_\Mcal^k \partial_u \psi_h(\bullet,u_h) = \partial_u \psi_h(\bullet,u_h)
	\end{align}
	under the assumption (B5).
    An explicit computation with \eqref{def:dual-stab} and \eqref{def:dual-variable} shows, for any $S \in \Fcal \setminus \Fcal_\mathrm{N}$, that
    \begin{align*}
        \|\sigma_S - \{\Pi_\Mcal^{k-1} \sigma_\Mcal\}_S \cdot \nu_S\|_{L^{r'}(S)}^{r'} = \|h_S^{-s}|[u_h]_S|^{r-2} [u_h]_S\|_{L^{r'}(S)}^{r'} = h_S^{-sr'}\|[u_h]_S\|^{r}_{L^r(S)}
    \end{align*}
    Since $-sr' + s/(r-1) = -s$,
	we obtain $\gamma_h(y) = s_h(u_h)$.
	This, the identities $\sigma_\Mcal \cdot \nabla_h u_h = W(\nabla_h u_h) + W^*(\sigma_\Mcal)$ and $\div_h y\,u_h = \psi_h(\bullet,u_h) + \psi_h^*(\bullet,\div_h y)$ from $\nabla_h u_h \in \partial W^*(\sigma_\Mcal)$ and from $u_h \in \partial_u \psi_h^*(\bullet,\div_h y)$ a.e.~in $\Omega$, and \eqref{eq:dELE} with the choice $v_h = u_h$ show
	\begin{align*}
		0 = \int_\Omega (W(\nabla_h u_h) + W^*(\sigma_\Mcal) + \psi_h(x,u_h) + \psi_h^*(x,\div_h y)) \d{x} + \frac{s_h(u_h)}{r} + \frac{\gamma_h(y)}{r'}&.
	\end{align*}
	Rearranging the terms on the right-hand side concludes the proof. 
\end{proof}

\section{Error analysis of LDG method}\label{sec:error-analysis}
In this section, we apply the duality relations in \Cref{sec:equivalence} to the error analysis. To establish error estimates, we assume for simplicity the explicit representation (B4) of the lower-order term. In this case,
\begin{align}\label{def:lower-order-terms}
	\psi_h(x,a) \coloneqq -f_h(x)\, a
\end{align}
with the $L^2$ orthogonal projection $f_h \coloneqq \Pi_\Mcal^{k} f \in P_{k}(\Mcal)$ of $f$ provides a suitable approximation satisfying (B5). Furthermore,
\begin{align*}
	E^*(\tau) &= -\int_\Omega W^*(\tau) \d{x} - \chi_{-f} (\div\,\tau) &&\text{for } \tau \in \Sigma,\\
	E^*_h(\tau) &= - \int_\Omega W^*(\tau_\Mcal) \d{x} - \chi_{-f_h}^*(\div_h \tau) - \gamma_h(\tau)/r' &&\text{for } \tau = (\tau_\Mcal, \tau_\mathcal{F}) \in Y
\end{align*}
with the indicator function $\chi_{g}(a) = 0$ if $a = g$ and $+\infty$ if $a \neq g$ for $a,g \in \mathbb{R}$.

\subsection{A~priori}
The ansatz space $V_h$ lacks trace degrees of freedom for the full consistency of the discrete gradient $\nabla_h$ from \eqref{def:discrete-gradient} with respect to discrete test functions.
Therefore, an additional tool is utilized in the a~priori error analysis.
\begin{lemma}[conforming companion]\label{lem:conforming_companion}
	There exists a linear bounded operator $J_h : V_h \to V$ such that
	any $v_h \in V_h$ satisfies $\Pi_\Mcal^{k} (v_h - J_h v_h) = 0$, $\Pi_S^{k} (J_h v_h - \{v_h\}_S) = 0$ for any $S \in \mathcal{F} \setminus \mathcal{F}_\mathrm{D}$.  Any $K \in \Mcal$ satisfies 
	\begin{align*}
		&h_K^{-1}\|v_h - J_h v_h\|_{L^p(K)}^p\\
		&\qquad + \|\nabla(v_h - J_h v_h)\|_{L^p(K)}^p \lesssim \sum_{S \in \mathcal{F}, S \cap K \neq \emptyset} h^{1-p}\|[v_h]_S\|_{L^p(S)}^p.
	\end{align*}
	In particular, $\nabla_h v_h = \Pi_\Mcal^{k-1} \nabla J_h v_h$.
\end{lemma}
\begin{proof}
	The explicit construction of $J_h$ utilizes well-understood
	averaging and bubble functions techniques, cf.~\cite{VeeserZanotti2018,ErnZanotti2020} for further details. The asserted bound is given in \cite{VeeserZanotti2018,ErnZanotti2020} for $p = 2$ and the general case follows from scaling arguments. Further details on $J_h$ are omitted.
    
	The $L^2$ orthogonality $v_h - J_h v_h \perp \div_\pw \Phi$ and $J_h v_h - \{v_h\}_S \perp [\Phi \cdot \nu_S]_S$ for any $S \in \mathcal{F}\setminus\mathcal{F}_\mathrm{D}$ shows that the right-hand side of \eqref{def:discrete-gradient} is equal to $(\nabla J_h v_h, \Phi)_{L^2(\Omega)}$ for any $\Phi \in P_{k-1}(\Mcal)^n$. This leads to
	$\nabla_h v_h = \Pi_\Mcal^{k-1} \nabla J_h v_h$.
\end{proof}

The subsequent theorem is the main result of this section.

\begin{theorem}[a~priori]\label{thm:error-estimate}
	Suppose (B1), (B4), \eqref{def:lower-order-terms}, and $\sigma \in W^{1,t}(\Omega)^n \cap \Sigma$. Then
	\begin{align*}
		E_h&(\mathrm{I}_h u) - \min E_h(V_h) \leq \int_\Omega (\sigma - \D W(\nabla_h \mathrm{I}_h u)) \cdot (\nabla u - \nabla_h \I_h u) \d{x} + s_h(\I_h u)/r\\
		&+ \int_\Omega \sigma \cdot (\nabla_h \mathrm{I}_h u - \nabla J_h \I_h u) \d{x} + \int_\Omega f (J_h \I_h u - \mathrm{I}_h u) \d{x} + \gamma_h(\mathrm{I}_h^* \sigma)/r'.
	\end{align*}
\end{theorem}

\begin{proof}
\Cref{lem:consistency} implies $\div_h \mathrm{I}_h^* \sigma = -f_h$ and so,
$E^*_h(\mathrm{I}_h^* \sigma) = E^*(\sigma) - \gamma_h(\mathrm{I}_h^* \sigma)/r'$.
This, \Cref{thm:weak-duality}, and \eqref{eq:strong-duality} reveal
	\begin{align}\label{ineq:proof-error-estimate-split}
		0 \leq E_h(\mathrm{I}_h u) - \min E_h(V_h) &\leq E_h(\mathrm{I}_h u) - E^*_h(\mathrm{I}_h^* \sigma)\nonumber\\
		&= E_h(\mathrm{I}_h u) - E(u) + \gamma_h(\mathrm{I}_h^* \sigma)/r'.
	\end{align}
    The convexity $0 \leq W(\nabla u) - W(\nabla_h \I_h u) - \D W(\nabla_h \mathrm{I}_h u) \cdot (\nabla u - \nabla_h \I_h u)$ a.e.~in $\Omega$ of $W$ provides
	\begin{align*}
		E_h(\I_h u) - E(u) \leq - \int_\Omega (\D W(\nabla_h \mathrm{I}_h u) \cdot (\nabla u - \nabla_h \I_h u) + f (u - \I_h u)) \d{x} + \frac{\s_h(\I_h u)}{r}.
	\end{align*}
	The combination of this with the Euler-Lagrange equations
	\begin{align*}
		\int_\Omega f (u - J_h \I_h u) \d{x} = \int_\Omega \sigma \cdot \nabla (u - J_h \I_h u) \d{x}.
	\end{align*}
results in the bound
	\begin{align}\label{ineq:proof-convergence-rates-primal}
		E_h(\I_h u) - E(u) \leq \int_\Omega (\sigma - \D W(\nabla_h \I_h u)) \cdot (\nabla u - \nabla_h \I_h u) \d{x} + \frac{1}{r}\s_h(\I_h u)&\nonumber\\
		+ \int_\Omega \sigma \cdot (\nabla_h \I_h u - \nabla J_h \I_h u) \d{x} + \int_\Omega f (J_h \I_h u - \mathrm{I}_h u) \d{x}&.
	\end{align}
	This and \eqref{ineq:proof-error-estimate-split} conclude the proof.
\end{proof}

Convergence rates in terms of the maximal mesh-size $h_{\max}$ can be derived from \Cref{thm:error-estimate} under suitable smoothness assumptions as follows.

\begin{corollary}[convergence rates]\label{cor:convergence-rates}
Suppose that the assumptions of \Cref{thm:error-estimate} hold and $\D W(\nabla_h \mathrm{I}_h u)$ is uniformly bounded in $L^{p'}(\Omega)^n$ independent of the mesh-size.
	If $u \in V \cap W^{k+1,\max\{p,r\}}(\Mcal)$ and $\sigma \in W^{1,t}(\Omega)^n \cap W^{k,\max\{p',r'\}}(\Mcal)^n$, then $$E_h(\I_h u) - \min E_h(V_h) \lesssim h_{\max}^{\ell}$$
	with $\ell \coloneqq \min\{k, (k+1)r-1-s, ((s+1)+(k-1)r)/(r-1)\}$.
\end{corollary}
We mention that $\ell \geq 0$ for $(k+1)r - 1 \geq s \geq - (k-1)r - 1$.
\begin{proof}
Standard arguments involving, e.g., the trace inequality and the approximation property of the $L^2$ projections lead to
	\begin{align}\label{ineq:proof-stab-conv}
		\s_h(\mathrm{I}_h u) \lesssim h_\mathrm{max}^{r-1-s}\|\nabla_\pw(u - \Pi_{\Mcal}^k u)\|_r^r &\lesssim h_\mathrm{max}^{(k+1)r-1-s}|u|_{W^{k+1,r}(\Mcal)}^r,\\
		\gamma_h(\mathrm{I}_h^* \sigma) \lesssim h_{\max}^{(s+1)/(r-1)}\|\nabla_\pw(\sigma - \Pi_{\Mcal}^{k-1}\sigma)\|^{r'}_{r'} &\lesssim h_{\max}^{((s+1)+(k-1)r)/(r-1)}|\sigma|_{W^{k,r'}(\Mcal)}^{r'}.\nonumber
	\end{align}
	From \Cref{lem:conforming_companion}, \Cref{lem:discrete-consistency-error}, and a triangle inequality, we deduce that
	\begin{align*}
		h_\mathrm{max}^{-1}\|\mathrm{I}_h u - J_h \mathrm{I}_h u\|_{p} &+ \|\nabla_\pw (\mathrm{I}_h u - J_h \mathrm{I}_h u)\|_{p}\\
		& + \|\nabla_h \I_h u - \nabla J_h \I_h u\|_{p} \lesssim h^k_\mathrm{max}|u|_{W^{k+1,p}(\Mcal)}.
	\end{align*}
	Since $f = -\div\,\sigma \in W^{k-1,p'}(\Mcal)$, this and the $L^2$ orthogonality $\nabla_h \I_h u - \nabla J_h \I_h u \perp P_{k-1}(\Mcal)^n$ and $\I_h u - J_h \mathrm{I}_h u \perp P_{k}(\Mcal)$ from \Cref{lem:conforming_companion} imply
	\begin{align}\label{ineq:proof-conv}
		\int_\Omega \sigma &\cdot (\nabla_h \I_h u - \nabla J_h \I_h u) \d{x} + \int_\Omega f (J_h \I_h u - \mathrm{I}_h u) \d{x}\nonumber\\
		&= \int_\Omega \big((1 - \Pi_\Mcal^{k-1}) \sigma \cdot (\nabla_h \I_h u - \nabla J_h \I_h u) + (1 - \Pi_\Mcal^k) f (J_h \I_h u - \mathrm{I}_h u)\big) \d{x}\nonumber\\
		&\lesssim h^{2k}_\mathrm{max}|u|_{W^{k+1,p}(\mathcal{M})}|\sigma|_{W^{k,p'}(\mathcal{M})}.
	\end{align}
	Under the smoothness assumptions of \Cref{cor:convergence-rates}, \Cref{lem:discrete-consistency-error} leads to $\|\nabla_\pw \mathrm{I}_h u - \nabla_h \mathrm{I}_h u\|_{p} \lesssim h_\mathrm{max}^k|u|_{W^{k+1,p}(\Mcal)}$. This and a triangle inequality prove
	\begin{align*}
		\|\nabla u - \nabla_h \mathrm{I}_h u\|_{p} \lesssim h_\mathrm{max}^k|u|_{W^{k+1,p}(\Mcal)}.
	\end{align*}
	Therefore, a H\"older inequality and the boundedness of $\sigma - \D W(\nabla_h \I_h u)$ in $L^{p'}(\Omega)$ by assumption provide
	\begin{align}\label{ineq:proof-conv-suboptimal}
		\int_\Omega (\sigma &- \D W(\nabla_h \I_h u)) \cdot (\nabla u - \nabla_h \I_h u) \d{x}\nonumber\\
		&\qquad\leq \|\sigma - \D W(\nabla_h \mathrm{I}_h u)\|_{p'}\|\nabla u - \nabla_h \mathrm{I}_h u\|_{p}\nonumber\\
		&\qquad\lesssim h_\mathrm{max}^k(\|\sigma\|_{p'} + 1)|u|_{W^{k+1,p}(\mathcal{M})}.
	\end{align}
	The combination of \eqref{ineq:proof-stab-conv}--\eqref{ineq:proof-conv-suboptimal}  with \Cref{thm:error-estimate} concludes the proof.
\end{proof}

\begin{remark}[balancing weights for stabilization]\label{rem:balancing-stab}
	To obtain balanced convergence rates for the stabilizations on the primal and dual level in \eqref{ineq:proof-stab-conv} under the smoothness assumptions of \Cref{cor:convergence-rates}, we can choose the parameter $s = (k+1)(r-2)+1$,
	\begin{align*}
		(k+1)r-1-s = 2k = ((s+1)+(k-1)r)/(r-1).
	\end{align*}
	This leads to squared convergence rates for the stabilizations in \eqref{ineq:proof-stab-conv}.
\end{remark}

\begin{remark}[choice of $s$]
Under the assumptions of \Cref{cor:convergence-rates},
the best possible rate is bounded by $k$
obtained for $r-k-2 \leq s \leq (k+1)(r-1)$. 
This includes the choice $r = p$ and $s = p-1$ as in \cite{BurmanErn2008}.
\end{remark}

Under additional structural assumptions on the energy density $W$, however, the convergence rates in \Cref{cor:convergence-rates} can be improved further.
Suppose that $p \geq 2$ and we refer to \Cref{remark:p_smaller_2} below for the case $1 < p < 2$. We consider the assumptions (B1)--(B4) from the introduction
and recall from \cite{GlowinskiMarrocco1975,Chow1989,CPlechac1997} that (B3) implies,
for all $\alpha,\beta \in L^p(\Omega)^n$,
\begin{align}\label{ineq:cc-continuous}
    \delta(\alpha,\beta) &\coloneqq \frac{\|\D W(\alpha) - \D W(\beta)\|_{p'}^2}{(1 + \|\alpha\|^p_{p} + \|\beta\|_{p}^p)^{(2-p')/p'}}\nonumber\\
    &\lesssim \int_\Omega (W(\beta) - W(\alpha) - \D W(\alpha) \cdot (\beta - \alpha)) \d{x}.
\end{align}

\begin{proposition}[convergence rates for degenerate convex minimization problems]\label{prop:conv-rates}
Suppose (B1)--(B4), $s = (k+1)(r-2)+1$, $u \in V \cap W^{k+1,\max\{p,r\}}(\Mcal)$,
    and $\sigma \in W^{1,t}(\Omega)^n \cap W^{k,\max\{p',r'\}}(\Mcal)^n$. Then
	\begin{align*}
		|E(u) - \min E_h(V_h)| + E_h(\mathrm{I}_h u) - \min E_h(V_h) 
        =O(h_\mathrm{max}^{2k}).
	\end{align*}
\end{proposition}
Before the proof is given, we recall a consequence of the growth (B2) of $W$.
\begin{remark}[boundedness of primal variable]\label{rem:bdd-primal-variable}
	On the continuous level, the lower growth in (B2) provides the uniform bound
	$\|\nabla u\|_p \lesssim 1$,
	cf.~\cite{CPlechac1997} for explicit constants. This, a triangle inequality, and \Cref{lem:discrete-consistency-error} imply $\|\nabla_h \mathrm{I}_h u\|_p \lesssim 1$.
	Furthermore, \cite[Lemma 2.1(a)]{CarstensenTran2021} provides
	$\|\sigma\|_{p'} + \|\D W(\nabla_h \I_h u)\|_{p'} \lesssim 1$.
\end{remark}
\begin{proof}[Proof of \Cref{prop:conv-rates}]
    Exchanging the roles of $\alpha$ and $\beta$ in \eqref{ineq:cc-continuous} followed by the 
    sum of the two resulting inequalities proves, for any $\alpha, \beta \in L^p(\Omega)^n$, that
    \begin{align}\label{ineq:cc-monotonicity}
        \delta(\alpha,\beta) \lesssim \int_\Omega (\D W(\alpha) - \D W(\beta)) \cdot (\alpha - \beta) \d{x}
    \end{align}
	The choice $\alpha \coloneqq \nabla u$ and $\beta \coloneqq \nabla_h \mathrm{I}_h u$ in \eqref{ineq:cc-monotonicity} and a H\"older inequality imply
	\begin{align*}
		\delta(\nabla u, \nabla_h \mathrm{I}_h u) \lesssim \|\sigma - \D W(\nabla_h \mathrm{I}_h u)\|_{p'}\|\nabla u - \nabla_h \mathrm{I}_h u\|_{p}.
	\end{align*}
Since
$\|\nabla u\|_{p} + \|\nabla_h \mathrm{I}_h u\|_{p} \lesssim 1$
	from \Cref{rem:bdd-primal-variable}, 
this shows
	\begin{align}\label{ineq:conv-sigma-interp}
		\|\sigma - \D W(\nabla_h \mathrm{I}_h u)\|_{p'} \lesssim \|\nabla u - \nabla_h \mathrm{I}_h u\|_{p} \lesssim h_\mathrm{max}^{k}|u|_{W^{k+1,p}(\Mcal)}.
	\end{align}
Therefore, we deduce from a H\"older inequality that
    \begin{align}\label{ineq:proof-conv-opt}
        \int_\Omega (\sigma - \D W(\nabla_h \I_h u)) \cdot (\nabla u - \nabla_h \I_h u) \d{x} \lesssim h_{\max}^{2k}|u|_{W^{k+1}(\Mcal)}^{2},
    \end{align}
    improving
    the convergence rates in \eqref{ineq:proof-conv-suboptimal}. This, \eqref{ineq:proof-stab-conv}--\eqref{ineq:proof-conv}, \Cref{rem:balancing-stab}, and \Cref{thm:error-estimate} prove
	\begin{align}\label{ineq:proof-conv-rate-1}
		E_h(\mathrm{I}_h u) - \min E_h(V_h) \lesssim h_\mathrm{max}^{2k}.
	\end{align}
It remains to control $|E(u) - E_h(I_h u)|$. Since
$E_h(\mathrm{I}_h u) - E(u) \lesssim h_\mathrm{max}^{2k}$ is known from \eqref{ineq:proof-convergence-rates-primal}--\eqref{ineq:proof-conv} and \eqref{ineq:proof-conv-opt}, it remains to control $E(u) - E_h(\mathrm{I}_h u)$.
The convexity of $W$ implies
$0 \leq W(\nabla_h \mathrm{I}_h u) - W(\nabla u) - \sigma \cdot (\nabla_h \mathrm{I}_h u - \nabla u)$ a.e.~in $\Omega$ and so  
	\begin{align*}
		&E(u) - E_h(\mathrm{I}_h u) \leq -\int_\Omega \sigma \cdot (\nabla_h \mathrm{I}_h u - \nabla u) \d{x} - \int_\Omega f(u - \I_h u) \d{x}\\
		&= - \int_\Omega \sigma \cdot (\nabla_h \mathrm{I}_h u - \nabla J_h \mathrm{I}_h u) \d{x} - \int_\Omega \sigma \cdot (\nabla J_h \mathrm{I}_h u - \nabla u) \d{x} - \int_\Omega f(u - \I_h u) \d{x}.
	\end{align*}
    The identity $(\sigma, \nabla J_h \mathrm{I}_h u - \nabla u)_{L^2(\Omega)} = (f, J_h \mathrm{I}_h u - u)_{L^2(\Omega)}$ from the Euler-Lagrange equations shows that the right-hand side is equal to the negative of the left-hand side of \eqref{ineq:proof-conv}, which implies $E(u) - E_h(\mathrm{I}_h u)
    =O(h_\mathrm{max}^{2k})$ 
    and so, $|E(u) - E_h(\mathrm{I}_h u)| =O(h_\mathrm{max}^{2k}) $.
    The combination of this with \eqref{ineq:proof-conv-rate-1} and a triangle inequality concludes the proof.
\end{proof}
\begin{remark}[significant choices of $r$]\label{rem:choice_s}
	For $r = 2$, $s = 1$ in \Cref{prop:conv-rates} is computationally attractive due to its quadratic structure. The case $r = p$ and $s = (k+1)(p-2) +1$ is of theoretical interest, where the regularity $u \in V \cap W^{k+1,p}(\Mcal)$ and $\sigma \in W^{1,1}(\Omega)^n \cap W^{k,p'}(\Mcal)^n$ is required in \Cref{prop:conv-rates}.
\end{remark}
\begin{remark}[global regularity]\label{rem:global_reg_normal_trace}
    Under the assumptions (B1)--(B4) and suitable geometric conditions on the intersection between Dirichlet and Neumann boundaries, \cite[Theorem 2.2]{Knees2008} provides the regularity $\sigma \in W^{1/2 - \delta, p'}(\Omega)$ for any $\delta > 0$. This and $\sigma \in \Sigma$ can guarantee the Lebesgue integrability of normal traces of $\sigma$ along $n-1$ dimensional hyperplanes without an extra assumption $\sigma \in W^{1,t}(\Omega)$.
    For the optimal design problem \cite{KohnStrang1986,BartelsC2008}, if $\Omega$ is a
    $C^{2,1}$ domain and $f \in H^1_0(\Omega)$, then $\sigma \in H^1(\Omega)$ \cite[Theorem 4.3]{CMueller02}.
\end{remark}
\begin{remark}[boundedness of discrete primal variable]\label{rem:bound-d-primal-variable}
	Note that the two-sided growth of $W$ in (B2) implies the two-sided growth
	\begin{align}\label{ineq:two-sided-dual}
		c_6|g|^{p'} - c_7 \leq W^*(g) \leq c_8|g|^{p'} + c_9 \quad\text{for any } g \in \R^n.
	\end{align}
	of $W^*$ with positive constants $c_6,c_8 > 0$ and non-negative constants $c_7,c_9 \geq 0$, cf., e.g., \cite[Lemma 2.1(b)]{CarstensenTran2021}. Under the assumptions of \Cref{prop:conv-rates}, the interpolation $\I_h^* \sigma$  satisfies $\div_h \I_h^* \sigma = - f_h$ from \Cref{lem:consistency}. Thus, \eqref{ineq:two-sided-dual} implies
	\begin{align*}
		-c_8\|\sigma\|_{p'}^{p'} - c_9|\Omega| + \gamma_h(\I_h^* \sigma) \leq E_h^*(\I_h^* \sigma) \leq E^*_h(y) \leq -c_6\|\sigma_\Mcal\|_{p'}^{p'} + c_7|\Omega|
	\end{align*}
	This, \eqref{ineq:proof-stab-conv}, and \Cref{rem:balancing-stab} show $\|\sigma_\Mcal\|_{p'} \lesssim 1$. Since $\nabla_h u_h \in \partial W^*(\sigma_\Mcal)$, $\|\nabla_h u_h\|_p \lesssim 1$ from \cite[Lemma 2.1(c)]{CarstensenTran2021}.
\end{remark}
\begin{remark}[convergence rates for the stress error]
	Suppose that the assumptions of \Cref{prop:conv-rates} hold.
	The choice $\alpha \coloneqq \nabla_h u_h$ and $\beta \coloneqq \nabla_h \mathrm{I}_h u$ in \eqref{ineq:cc-continuous} proves
	\begin{align*}
		\delta(\nabla_h u_h,\nabla_h \mathrm{I}_h u) \lesssim \int_\Omega (W(\nabla_h \mathrm{I}_h u) - W(\nabla_h u_h) - \D W(\nabla_h u_h) \cdot \nabla_h(\mathrm{I}_h u - u_h)) \d{x}.
	\end{align*}
	This, the discrete Euler-Lagrange equations \eqref{eq:dELE}, and \Cref{cor:convergence-rates} imply
	\begin{align}\label{ineq:err-quantity-conv}
		 \delta(\nabla_h u_h,\nabla_h \mathrm{I}_h u) \lesssim E_h(\mathrm{I}_h u) - E_h(u_h) \lesssim h_{\mathrm{max}}^{2k}.
	\end{align}
	Since $\nabla_h u_h$ and $\nabla_h \I_h u$ are uniformly bounded in $L^{p}(\Omega)^n$ from \Cref{rem:bdd-primal-variable} and \Cref{rem:bound-d-primal-variable}, $\|\sigma - \D W(\nabla_h u_h)\|_{p'} \lesssim h_{\max}^k$.
\end{remark}
\begin{remark}[strongly monotone]
	Assume (B2)--(B3), and
	\begin{align}
		c_9^{-1}|a - b|^p &\leq W(b) - W(a) - \D W(a) \cdot (b - a) \quad\text{for any } a,b \in \mathbb{R}^n\label{ineq:cc-primal}
	\end{align}
    and a positive constant $c_9 > 0$.
	Then the abstract error quantity in \eqref{ineq:cc-continuous} can be replaced by
	\begin{align*}
		\delta(\alpha,\beta) \coloneqq \frac{\|\D W(\alpha) - \D W(\beta)\|_{p'}^2}{(1 + \|\alpha\|^p_{p} + \|\beta\|_{p}^p)^{(2-p')/p'}} + \|\alpha - \beta\|_p^p.
	\end{align*}
    The strong convexity of the energy leads to a unique minimizer $u = \arg\min E(V)$.
	From \eqref{ineq:err-quantity-conv} and a triangle inequality, we deduce the convergence rates $\|\nabla u - \nabla_h u_h\|_p \lesssim h_{\max}^{2k/p}$, improving the rates $h_{\max}^{k/(p-1)}$ over the literature on nonconforming methods of arbitrary order \cite{DiPietroDroniou2017-II,DroniouEymardGallouet2018}.
	The conditions (B2)--(B3) and \eqref{ineq:cc-primal} are satisfied, e.g., in the $p$-Laplace problem.
    Then the convergence rates $h_{\max}^{\min\{2,p'\}}$ for the LDG method of \cite{BurmanErn2008} have been derived in \cite{LDG2014} under regularity assumptions based on the natural distance \cite{EbmeyerLiuSteinhauer2005}. The latter can be guaranteed under natural assumptions on the domain and right-hand side.
\end{remark} 

%

\begin{remark}[$1 < p \leq 2$]\label{remark:p_smaller_2}
    If $1 < p \leq 2$,
    we assume (B3) with $d(a,b) \coloneqq |\D W(a) - \D W(b)|^{p'}$, which implies \eqref{ineq:cc-monotonicity} for the error quantity
    \begin{align*}
    	\delta(\alpha,\beta) \coloneqq \|\D W(\alpha) - \D W(\beta)\|_{p'}^{p'}.
    \end{align*}
	The choice $\alpha \coloneqq \nabla u$ and $\beta \coloneqq \nabla_h \mathrm{I}_h u$ in \eqref{ineq:cc-monotonicity} and a H\"older inequality lead to
	\begin{align}\label{ineq:conv-sigma-p-sub-2}
		\|\sigma - \D W(\nabla_h \mathrm{I}_h u)\|_{p'} \lesssim \|\nabla u - \nabla \mathrm{I}_h u\|_{p}^{p-1}.
	\end{align}
	Suppose that $u \in V \cap W^{k+1,\max\{p,r\}}(\Mcal)$ and $\sigma \in W^{1,1}(\Omega)^n \cap W^{k,\max\{p',r'\}}(\Mcal)^n$, then \eqref{ineq:conv-sigma-p-sub-2} and a H\"older inequality imply
	\begin{align*}
		\int_\Omega (\sigma - \D W(\nabla_h \I_h u)) \cdot (\nabla u - \nabla_h \I_h u) \d{x} \leq \|\nabla u - \nabla_h \I_h u\|_p^p \lesssim h_{\max}^{kp}.
	\end{align*}
	The combination of this with \eqref{ineq:proof-stab-conv}--\eqref{ineq:proof-conv}, and \Cref{thm:error-estimate} concludes
	\begin{align}\label{ineq:conv-rate-p-sub-2}
		|E(u) - \min E_h(V_h)| + E_h(\mathrm{I}_h u) - \min E_h(V_h) \lesssim h_{\max}^{kp}
	\end{align}
	for $kp(r-1) - (k-1)r - 1 \leq s \leq (k+1)r - kp - 1$.
	This, $\delta(\nabla_h u_h, \nabla_h \I_h u) \lesssim E_h(\I_h u) - E_h(u_h)$ from \eqref{ineq:err-quantity-conv}, \eqref{ineq:conv-sigma-p-sub-2}, and a triangle inequality conclude
	\begin{align*}
		\|\sigma - \D W(\nabla_h u_h)\|_{p'} \lesssim h_{\max}^{k(p-1)}.
	\end{align*}
	If we additionally assume that any $a,b \in \mathbb{R}^n$ satisfy
	\begin{align*}
		\frac{|a-b|^2}{1 + |a|^{2-p} + |b|^{2-p}} \lesssim W(b) - W(a) - \D W(a) \cdot (b-a),
	\end{align*}
	then any $\alpha, \beta \in L^p(\Omega)^n$ satisfy \eqref{ineq:cc-continuous} with
	\begin{align*}
		\delta(\alpha,\beta) \coloneqq \|\D W(\alpha) - \D W(\beta)\|_p^p + \frac{\|\alpha - \beta\|^2}{(1 + \|\alpha\|_p^p + \|\beta\|_p^p)^{(2-p)/p}}.
	\end{align*}
	From \eqref{ineq:err-quantity-conv} and \eqref{ineq:conv-rate-p-sub-2}, we infer
    $\delta(\nabla_h u_h,\nabla_h \mathrm{I}_h u) \lesssim E_h(\mathrm{I}_h u) - E_h(u_h) = O(h_{\max}^{kp})$.
	The uniform boundedness of $\|\nabla_h u_h\|_{p} \lesssim 1$ follows from the arguments of \Cref{rem:bound-d-primal-variable}. This, $\delta(\nabla_h u_h,\nabla_h \mathrm{I}_h u) \lesssim h_{\max}^{kp}$, and a triangle inequality imply the convergence rates $\|\nabla u - \nabla_h u_h\|_p \lesssim h_\mathrm{max}^{kp/2}$. This recovers the result of \cite{Tran2024} for hybridizable and \cite{Chow1989} for conforming methods for strongly monotone problems.
\end{remark}

\subsection{A~posteriori}\label{sec:a-posteriori}
A $\Sigma$-conforming approximation of the dual variable in the Raviart-Thomas finite element space is constructed by direct prescription of the degrees of freedom. This provides an alternative to equilibrium techniques \cite{LuceWohlmuth2004,BraessSchoeberl2008,ErnVohralik2015} with solving local problems.
For the sake of brevity, we assume that $\Mcal$ is a regular triangulation into simplices (without hanging nodes) and refer to \cite{ErnVohralik2009} and \cite[Section 5]{Tran2024} for further details if $\Mcal$ is a polytopal mesh.

Recall $y = (\sigma_\Mcal, \sigma_\mathcal{F}) \in Y$ from \eqref{def:dual-variable}. 
Let $\sigma_\RT \in \mathrm{RT}_{k}(\Mcal) \cap \Sigma$ be the unique Raviart-Thomas finite element function with
\begin{align}\label{def:sigma-h}
	\Pi_\Mcal^{k-1} \sigma_\RT = \Pi_\Mcal^{k-1} \sigma_\Mcal \quad\text{and}\quad \Pi_S^k \sigma_\RT = \Pi_S^k \sigma_\mathcal{F} \text{ for any } S \in \mathcal{F}\setminus\mathcal{F}_\mathrm{N}.
\end{align}

\begin{theorem}[post-processing]\label{thm:post-processing}
	Let $\Mcal$ be a regular triangulation of $\Omega$ into simplices.
	Suppose (B1) and $\psi_h(x,\bullet) \in C^1(\R)$ for a.e.~$x \in \Omega$. Then $\sigma_\RT \in \RT_k(\Mcal) \cap \Sigma$ from \eqref{def:sigma-h} satisfies 
    $\div\,\sigma_\RT = \div_h \sigma_h = \Pi_\Mcal^{k} \partial_u \psi_h(\bullet,u_h).$
\end{theorem}

\begin{proof}
Integration by parts, \eqref{def:sigma-h}, and \eqref{def:div-rec} show, for any $\phi \in P_k(\Mcal)$, that
	\begin{align*}
		\int_\Omega \div\,\sigma_\RT \phi \d{s} &= -\int_\Omega \sigma_\RT \cdot \nabla_\pw \phi \d{x} + \sum_{S \in \mathcal{F}\setminus\mathcal{F}_\mathrm{N}} \int_S [\phi]_S \sigma_\RT\cdot \nu_S \d{s}\\
		&= -\int_\Omega \sigma_\Mcal \cdot \nabla_\pw \phi \d{x} + \sum_{S \in \mathcal{F}\setminus\mathcal{F}_\mathrm{N}} \int_S [\phi]_S \sigma_S \d{s} = \int_\Omega \div_h \sigma_h \phi \d{x}.
	\end{align*}
	Since $\div_h \sigma_h = \Pi_\Mcal^{k} \partial_u \psi_h(\bullet,u_h)$ by \eqref{eq:div_h_sigma_h}, this concludes the proof.
\end{proof}
\begin{remark}[a~posteriori error control]
Suppose that the lower-order term $\psi$ has the explicit representation (B4) and the assumptions of \Cref{thm:post-processing} hold.
	Furthermore, we assume that $f \in P_k(\Mcal)$ is a piecewise polynomial.
	Given a conforming postprocessing $v_C \in V$, then the energy error $E(v_C) - E(u)$ can be bounded by
	\begin{align*}
		E(v_C) - E(u) \leq E(v_C) - E^*(\sigma_\RT).
	\end{align*}
In the numerical examples below, $v_C$ is obtained from the discrete minimizer $u_h$ of $E_h$ in $V_h$ by nodal averaging in the conforming subspace $V_h \cap V$. If the energy density $W$ satisfies further structural properties, e.g., \eqref{ineq:cc-continuous}, then the Euler-Lagrange equations and the previously displayed formula imply
	\begin{align}\label{ineq:a-posteriori}
		\begin{split}
			\delta(\nabla u, \nabla v_C) &\lesssim \int_\Omega (W(\nabla v_C) - W(\nabla u) - \sigma \cdot \nabla (v_C - u)) \d{x}\\
			&= E(v_C) - E(u) \leq E(v_C) - E^*(\sigma_\RT) \eqqcolon \eta.
		\end{split}
	\end{align}
If $f$ is not piecewise polynomial, then additional data oscillation arises in \eqref{ineq:a-posteriori}. However, this additional error term can be computed explicitly \cite[Remark 5.3]{Tran2024}.
\end{remark}

\section{Extension to hybridizable method}\label{sec:hybrid}
In this section, we briefly extend the analysis of \Cref{sec:error-analysis} to a hybridizable method using the techniques of \cite{Tran2024}. For the sake of simplicity, we retain the notation of \Cref{sec:discretization} on the discrete level.
Given $k \geq 1$, let
\begin{align*}
	V_h \coloneqq P_k(\Mcal) \times P_k(\mathcal{F}\setminus\mathcal{F}_\mathrm{D})
\end{align*}
denote the discrete ansatz space. Given $v_h = (v_\Mcal,v_\mathcal{F}) \in V_h$, the discrete gradient $\nabla_h v_h \in P_{k-1}(\Mcal)^n$ of $v_h$ is the unique solution to
\begin{align*}
	\int_\Omega \nabla_h v_h \cdot \Phi \d{x} &= -\int_\Omega v_\Mcal \,\div_\pw \Phi \d{x} + \sum_{S \in \mathcal{F}\setminus\mathcal{F}_\mathrm{D}} \int_S v_S [\Phi \cdot \nu_S] \d{s}
\end{align*}
for any $\Phi \in P_{k-1}(\Mcal)$,
where $v_S = v_\mathcal{F}|_S$ abbreviates the restriction of $v_\mathcal{F}$ along the side $S$. The discrete problem minimizes
\begin{align}
	E_h(v_h) \coloneqq \int_\Omega (W(\nabla_h v_h) + \psi_h(\bullet,v_\Mcal)) \d{x} + \s_h(v_h)/r
\end{align}
among $v_h = (v_\Mcal, v_\mathcal{F}) \in V_h$
with the stabilization
\begin{align*}
	\s_h(v_h) \coloneqq \sum_{K \in \Mcal} \sum_{S \in \mathcal{F}(K)} h_S^{-s} \int_S T_{K,S} v_h (v_S - v_K) \d{s}
\end{align*}
and $T_{K,S} v_h \coloneqq |v_S - v_K|^{r-2}(v_S - v_K)$ for any $K \in \Mcal$, $S \in \mathcal{F}$.
The corresponding dual problem is \eqref{def:dual-discrete-energy}, 
but with the stabilization
\begin{align*}
	\gamma_h(\tau) \coloneqq \sum_{K \in \Mcal} \sum_{S \in \mathcal{F}(K)} h^{s/(r-1)}\|\tau_S - \Pi_\Mcal^{k-1}\tau_K \cdot \nu_S\|^{r'}_{L^{r'}(S)}
\end{align*}
for any $\tau = (\tau_\Mcal, \tau_\mathcal{F}) \in Y$
instead of \eqref{def:dual-stab}
to reflect the hydridization of the ansatz space.
Here, $\tau_K = \tau_\Mcal|_K$ is the restriction of $\tau_\Mcal$ to $K$.
\begin{remark}[suboptimal polynomial consistency]\label{rem:LS-stab}
	For the Lehrenfeld-Sch\"oberl stabilization, we can use the discrete ansatz space $P_k(\Mcal) \times P_{k-1}(\mathcal{F}\setminus\mathcal{F}_\mathrm{D})$, reducing the computational cost of the method. However, the analysis of this section does not carry over because \eqref{def:div-rec} forfeits to hold.
\end{remark}

The following  \Cref{thm:duality-hybrid} allows for the extension of {\em all results of \Cref{sec:error-analysis}} to the hybrid method this section.

\begin{theorem}[duality of hybridizable methods]\label{thm:duality-hybrid}
	It holds $\sup E^*_h(Y) \leq \min E_h(V_h)$; 
	(B1) and (B5) imply
    $\max E_h^*(Y) = \min E_h(V_h)$.
\end{theorem}

\begin{proof}
For any $v_h = (v_\Mcal, v_\mathcal{F}) \in V_h$ and $\tau = (\tau_\Mcal, \tau_\mathcal{F}) \in Y$, the proof departs from the integration by parts formula
	\begin{align}\label{eq:dip-hho}
		\int_\Omega \tau_\Mcal \cdot \nabla_h v_h \d{x} &= \int_\Omega \Pi_{\Mcal}^{k-1} \tau_\Mcal \cdot \nabla_h v_h \d{x} = - \int_\Omega v_\Mcal \cdot \div_h \tau \d{x}\nonumber\\ 
		&\qquad+ \sum_{S \in \mathcal{F}\setminus\mathcal{F}_\mathrm{D}} \int_S (v_S - \{v_\Mcal\}_S) \cdot [\Pi_{\Mcal}^{k-1} \tau_\Mcal \cdot \nu_S]_S \d{s}\nonumber\\
		&\qquad + \sum_{S \in \mathcal{F}\setminus\mathcal{F}_\mathrm{N}} \int_S [v_\Mcal]_S \cdot (\tau_S - \{\Pi_{\Mcal}^{k-1} \tau_\Mcal \cdot \nu_S\}_S) \d{s}
	\end{align}
This follows from arguments similar to \cite[Lemma 3.2]{Tran2024}. 
Furthermore, the final two sums on the right-hand side can be rewritten as
	\begin{align*}
		- \sum_{K \in \mathcal{M}} \sum_{S \in \mathcal{F}(K)} (\nu_S \cdot \nu_K) \int_S (v_S - v_K) \cdot (\tau_S - \Pi_\Mcal^{k-1}\tau_K \cdot \nu_S) \d{s},
	\end{align*}
	cf.~\cite[Proof of Theorem 3.1]{Tran2024}. This, \eqref{eq:dip-hho}, the H\"older inequality, and $\tau_\Mcal \cdot \nabla_h v_h \leq W(\nabla_h v_h) + W^*(\tau_\Mcal)$ as well as $\div_h \tau \, v_h \leq \psi_h(\bullet,v_h) + \psi_h^*(\bullet,\div_h \tau_h)$ a.e.~in $\Omega$ conclude $\sup E^*_h(Y) \leq \min E_h(V_h)$.
    
	Assuming the differentiability of $W$ and $\psi_h$, we can define the stress variable $y = (\sigma_\Mcal, \sigma_\mathcal{F}) \in Y$ as
	\begin{align}\label{def:dual-variable-hybrid}
		\begin{split}
			\sigma_\Mcal &\coloneqq \D W(\nabla_h u_h),\\
			\sigma_\mathcal{F}|_S &\coloneqq \begin{cases}
				\{\Pi_\Mcal^{k-1} \sigma_\Mcal\}_S \cdot \nu_S + h_S^{-s}(T_{K_+,S} u_h - T_{K_-,S} u_h)/2 &\mbox{if } S \in \mathcal{F}(\Omega),\\
				\Pi_\Mcal^{k-1} \sigma_\Mcal \cdot \nu_S + h_S^{-s}T_{K_+,S} u_h &\mbox{if } S \in \mathcal{F}_\mathrm{D}.
			\end{cases}
		\end{split} 
	\end{align}
The computations in \cite[Corollary 5.1 and Lemma 5.2]{Tran2024} carry over and show that
	\begin{align}
		[\sigma_\Mcal \nu_S]_S &= - \sum_{K \in \Mcal, S \in \mathcal{F}(K)} h_S^{-s} T_{K,S} u_h \quad\text{for any } S\in\mathcal{F}\setminus\mathcal{F}_\mathrm{N},\label{eq:jump-dual-var}\\
		\div_h y &= \partial_u \psi_h(\bullet,u_\Mcal).\nonumber
	\end{align}
	Thus, the discrete Euler-Lagrange equations and \eqref{eq:dip-hho} imply
	\begin{align}\label{eq:proof-strong-dual-hybrid}
		- \int_\Omega \div_h y\, u_\Mcal \d{x} & = - \int_\Omega \partial_u \psi_h(\bullet,u_\Mcal) u_\Mcal \d{x}\nonumber\\
		& = \int_\Omega \sigma_\Mcal \cdot \nabla_h u_h \d{x} + \s_h(u_h).
	\end{align}
	An explicit computation with the definition of $y$ in \eqref{def:dual-variable-hybrid} and \eqref{eq:jump-dual-var} show, for any $K \in \Mcal$ and $S \in \mathcal{F}(K)$, that
	\begin{align*}
		\sigma_S - \Pi_K^{k-1} \sigma_K \cdot \nu_S &= -(\nu_S \cdot \nu_K)[\Pi_\Mcal^{k-1} \sigma_\Mcal]_S/2 + h_S^{-s}(T_{K_+,S} u_h - T_{K_-,S} u_h)/2\\
        &= (\nu_S \cdot \nu_K)h^{-s} T_{K,S} u_h.
	\end{align*}
	for interior sides $S \in \mathcal{F}(\Omega)$ and
	\begin{align*}
		\sigma_S - \Pi_K^{k-1} \sigma_K \cdot \nu_S = h_S^{-s} T_{K,S} u_h
	\end{align*}
	for boundary sides $S \in \mathcal{F}(\partial \Omega)$.
	This and $s/(r-1) - sr' = -s$ imply
	\begin{align}\label{eq:stab-hybrid-strong-dual}
		\gamma_h(y) = \s_h(u_h).
	\end{align}
	From \eqref{eq:proof-strong-dual-hybrid}--\eqref{eq:stab-hybrid-strong-dual}, $\sigma_\Mcal \cdot \nabla_h u_h = W(\nabla_h u_h) + W^*(\sigma_\Mcal)$, and $\div_h y \,u_\Mcal = \psi_h(\bullet,u_\Mcal) + \psi_h^*(\bullet,\div_h y)$ a.e.~in $\Omega$, we conclude $E_h^*(y) = E_h(u_h)$ and thus, $\max E_h^*(Y) = \min E_h(V_h)$.
\end{proof}

\section{Numerical examples}\label{sec:numerical_examples}
This section tests the performance of the a~posteriori error 
control \eqref{ineq:a-posteriori} in three numerical benchmarks 
in the L-shaped domain $\Omega \coloneqq (-1,1)^2 \setminus ([0,1) \times (-1,0])$ with constant right-hand side $f \equiv 1$.
The initial triangulation in all benchmarks is displayed in 
\Cref{fig:volume_fraction}(a).
The computer experiments are carried out on regular triangulations into \emph{simplices}.

\subsection{Adaptive mesh-refining algorithm}
Since $f$ is constant, the a~posteriori error estimator \eqref{ineq:a-posteriori} applies without data oscillation. The following localization of the right-hand side of \eqref{ineq:a-posteriori} was discussed in \cite{BartelsKaltenbach2023}.
An integration by parts with $\div \sigma_\RT = - f$ implies
\begin{align*}
	\eta = E(v_C) - E^*(\sigma_\RT) = \int_\Omega (W(\nabla v_C) - \sigma_\RT \cdot \nabla v_C + W^*(\sigma_\RT)) \d{x} \eqqcolon \sum_{K \in \Mcal} \eta(K)
\end{align*}
with the local refinement indicator
\begin{align}\label{def:refinement-indicator}
	\eta(K) \coloneqq \int_K (W(\nabla v_C) - \sigma_\RT \cdot \nabla v_C + W^*(\sigma_\RT)) \d{x}.
\end{align}
The Fenchel-Young inequality $W(\nabla v_C) - \sigma_\RT \cdot \nabla v_C + W^*(\sigma_\RT) \geq 0$ holds pointwise a.e.~in $\Omega$, whence $\eta(K) \geq 0$.
Adaptive computations utilize the refinement indicator \eqref{def:refinement-indicator} in the standard adaptive mesh-refining loop \cite{Doerfler1996,CarstensenFeischlPagePraetorius2014} with the D\"orfler marking strategy, i.e., at each refinement step, a subset $\mathfrak{M} \subset \mathcal{M}$ with minimal cardinality is selected such that
\begin{align*}
	\sum\nolimits_{K \in \mathcal{M}} \eta(K) \leq \frac{1}{2}\sum\nolimits_{K \in \mathfrak{M}} \eta(K).
\end{align*}
The convergence history plots display the a~posteriori error estimator $E(v_C) - E^*(\sigma_\RT)$ against the number of degrees of freedom $\mathrm{ndof}$ in a log-log plot.
(Recall the scaling $\mathrm{ndof} \approx h^{-2}_\mathrm{max}$ for uniform meshes.)
Solid lines indicate adaptive, while dashed lines are associated with uniform mesh refinements.
All plotted adaptive meshes are generated with the polynomial degree $k = 2$.

The discrete minimization problem $\min E_h(V_h)$ from \eqref{def:discrete_energy} is solved by an iterative solver \texttt{fminunc} from the MATLAB standard library in an extension of the data structures and the short MATLAB programs \cite{AlbertyCFunken1999}.
The first and (piecewise) second derivatives of $W$ have been provided for 
the trust-region quasi-Newton scheme with \texttt{MaxIterations} = $10^3$, while
\texttt{FunctionTolerance},  
\texttt{OptimalityTolerance}, and 
\texttt{tepTolerance}
in \texttt{fminunc} are set to 
$10^{-15}$.
The numerical integration of piecewise polynomials is carried out exactly.

For non-polynomial functions such as $W(\nabla_h v_h)$ with $v_h \in V_h$, the number of chosen quadrature points allows for exact integration of polynomials of degree at most $2pk + 1$ with the growth $p$ of $W$ and the polynomial order $k$ of the discretization. On the initial triangulation, the starting point for $\texttt{fminunc}$ is zero, while the conforming postprocessing $v_C$ 
initializes the starting point for the finer mesh. 

\begin{figure}[ht]
	\begin{minipage}[b]{0.475\textwidth}
		\centering
		\includegraphics[height=4.6cm]{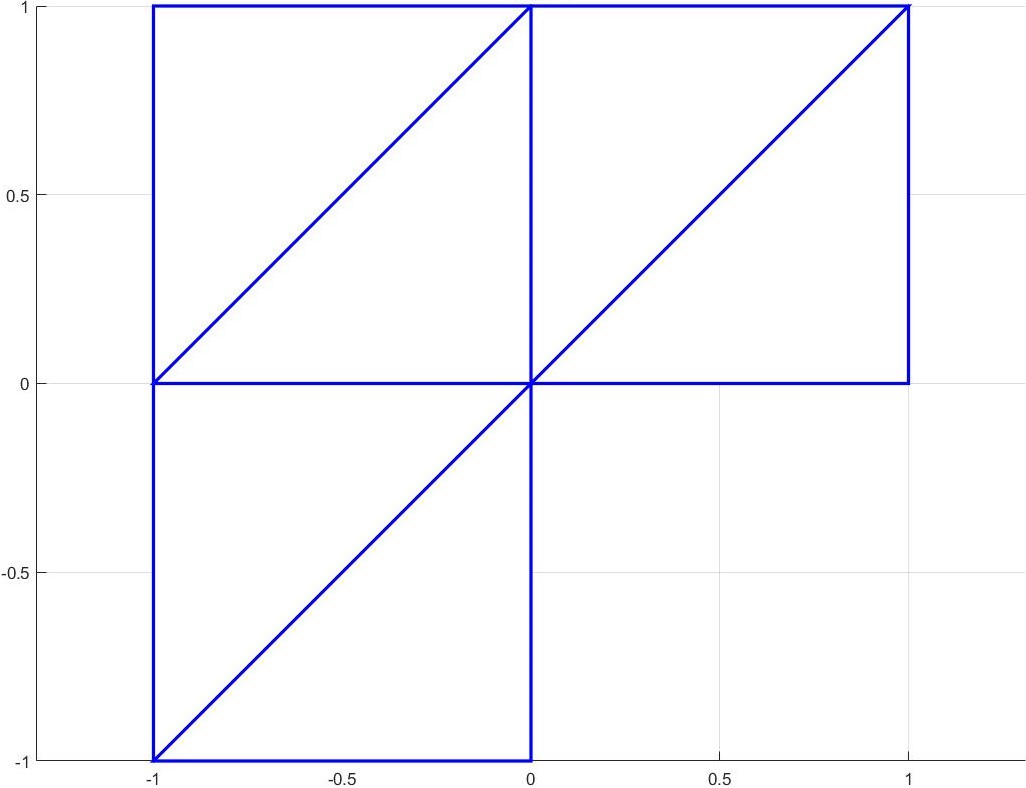}
	\end{minipage}\hfill
	\begin{minipage}[b]{0.475\textwidth}
		\centering
		\includegraphics[height=4.6cm]{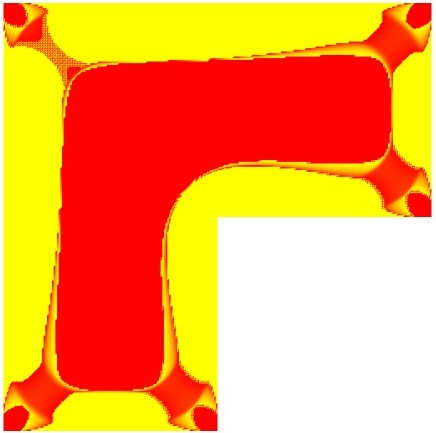}
	\end{minipage}
	\captionsetup{width=1\linewidth}
	\caption{(a) Initial triangulation of the L-shaped domain into 6 triangles and (b) material distribution in the optimal design problem of \Cref{sec:odp}}
	\label{fig:volume_fraction}
\end{figure}
\begin{figure}[ht]
	\begin{minipage}[b]{0.475\textwidth}
		\centering
		\includegraphics[height=5cm]{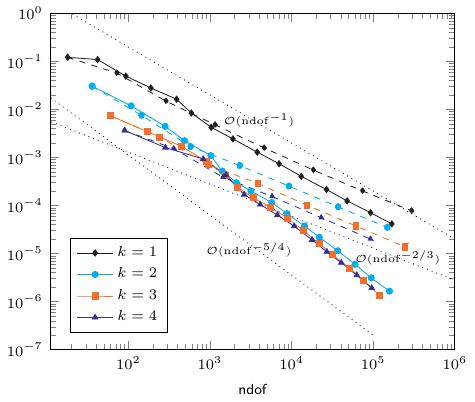}
	\end{minipage}\hfill
	\begin{minipage}[b]{0.475\textwidth}
		\centering
		\includegraphics[height=5cm]{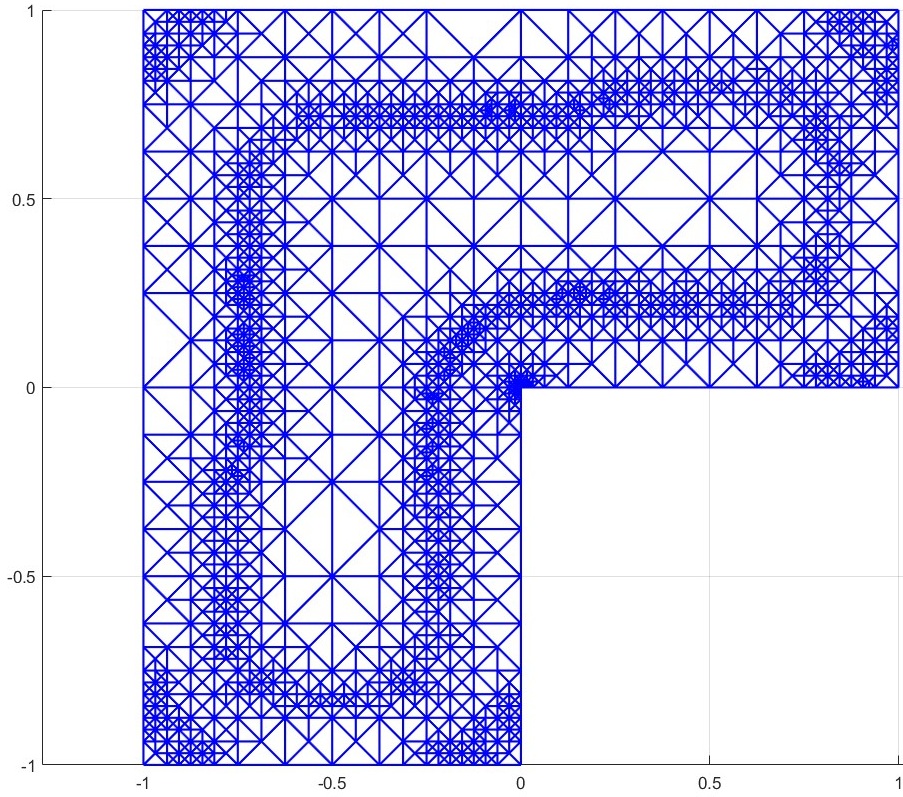}
	\end{minipage}
	\captionsetup{width=1\linewidth}
	\caption{(a) Convergence history plot of $\eta$ for $k = 1, \dots, 4$ and (b) adaptive triangulation in the optimal design problem of \Cref{sec:odp}}
	\label{fig:odp-conv}
\end{figure}

\subsection{Optimal design problem}\label{sec:odp}
This model problem seeks the optimal distribution of two materials with fixed amounts to fill a given domain for maximal torsion stiffness \cite{KohnStrang1986,BartelsC2008}. Given parameters $0 < t_1 < t_2$ and $0 < \mu_1 < \mu_2$ with $t_1 \mu_2 = \mu_1 t_2$, the energy density $W(a) \coloneqq w(|a|)$, $a \in \mathbb{R}^2$, with
\begin{align*}
	w(t) \coloneqq \begin{cases}
		\mu_2 t^2/2 &\mbox{if } 0 \leq t \leq t_1,\\
		t_1\mu_2(t - t_1/2) &\mbox{if } t_1 \leq t \leq t_2,\\
		\mu_1 t^2/2 + t_1\mu_2(t_2/2 - t_1/2) &\mbox{if } t_2 \leq t
	\end{cases}
\end{align*}
satisfies \eqref{ineq:a-posteriori} with $\delta(\alpha,\beta) \coloneqq \|\D W(\alpha) - \D W(\beta)\|^2_2$ for any $\alpha,\beta \in L^2(\Omega)^2$. Therefore, the energy error $E(v_C) - E^*(\sigma_\RT)$ provides an upper bound for the stress error $\|\sigma - \D W(\nabla v_C)\|^2_2$.
This benchmark considers the parameters $\mu_1 = 1$, $\mu_2 = 2$, $t_1 = \sqrt{2\lambda\mu_1/\mu_2}$ for $\lambda = 0.0145$, $t_2 = \mu_2 t_1/\mu_1$ from \cite{BartelsC2008}, the input $r = 2$, $s = 1$ for the stabilization $\s_h$, and $\Gamma_\mathrm{D} = \partial \Omega$.

The approximated material distribution in the  adaptive computation with $k = 1$ is displayed in \Cref{fig:volume_fraction}(b) using  volume fraction plot
\cite[Section 5]{BartelsC2008}.
On uniform meshes, a convergence rate  $2/3$ for $\eta$ is observed 
in \Cref{fig:odp-conv}(a).
The adaptive algorithm refines towards the singularity at the origin and the transition layer in \Cref{fig:odp-conv}(b).
This leads to improved convergence rates for $\eta$ although the improvements appear marginal for higher polynomial degrees.

\begin{figure}[ht]
	\begin{minipage}[b]{0.475\textwidth}
		\centering
		\includegraphics[height=5cm]{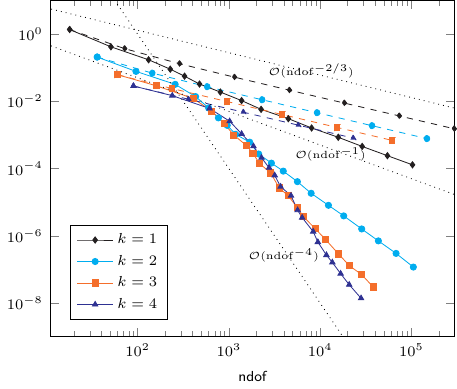}
	\end{minipage}\hfill
	\begin{minipage}[b]{0.475\textwidth}
		\centering
		\includegraphics[height=5cm]{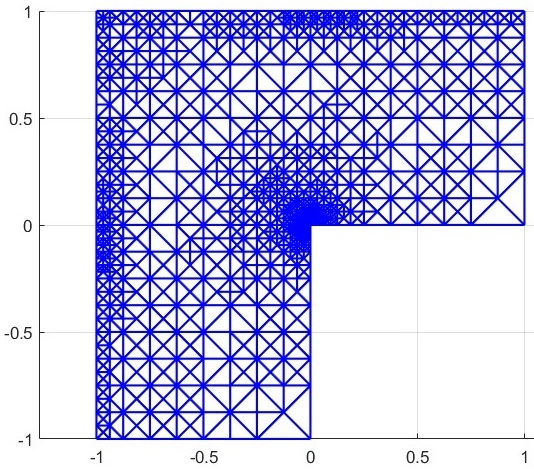}
	\end{minipage}
	\captionsetup{width=1\linewidth}
	\caption{(a) Convergence history plot of $\eta$ for $k = 1, \dots, 4$ and (b) adaptive triangulation in the $4$-Laplace problem of \Cref{sec:plaplace}}
	\label{fig:plaplace_conv}
\end{figure}

\begin{figure}[ht]
	\begin{minipage}[b]{0.475\textwidth}
		\centering
		\includegraphics[height=5cm]{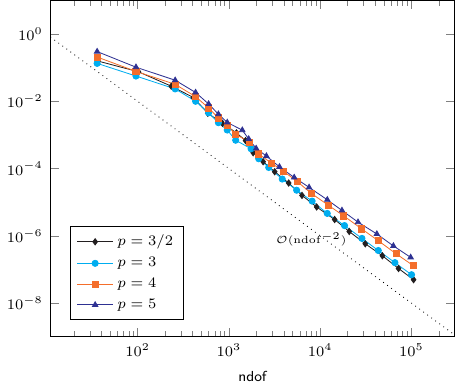}
	\end{minipage}\hfill
	\begin{minipage}[b]{0.475\textwidth}
		\centering
		\includegraphics[height=5cm]{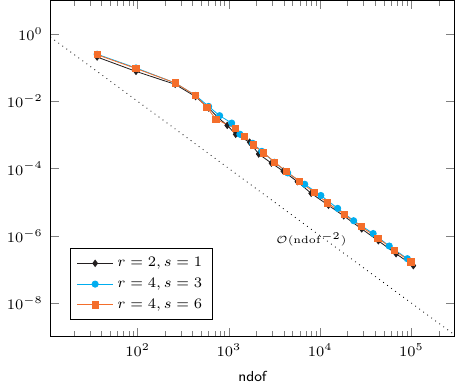}
	\end{minipage}
	\captionsetup{width=1\linewidth}
	\caption{Convergence history plot of $\eta$ for (a) fixed $k = 2$ and various $p$ and (b) fixed $k = 2, p = 4$ and various $r,s$ in adaptive computations}
	\label{fig:plaplace_comp}
\end{figure}

\subsection{$4$-Laplace problem}\label{sec:plaplace}
In this benchmark, we consider the $4$-Laplace problem with $W(a) \coloneqq |a|^4/4$
for any $a \in \mathbb{R}^2$ in $\Omega$ and Dirichlet boundary $\Gamma_\mathrm{D} \coloneqq (\{0\} \times [-1,0] \cup [0,1] \times \{0\})$.
Since $W$ satisfies (B3) and \eqref{ineq:cc-primal}, \eqref{ineq:a-posteriori} holds with
\begin{align*}
	\delta(\nabla u, \nabla v_C) \coloneqq \|\nabla(u - v_C)\|^4_4 + \frac{\|\D W(\nabla u) - \D W(\nabla v_C)\|^2_{4/3}}{(1 + \|\nabla u\|_4^4 + \|\nabla v_C\|_4^4)^{1/2}}.
\end{align*}
Furthermore, $\D W$ is strongly monotone w.r.t.~the quasi norm \cite{BarrettLiu1993,BarrettLiu1994,DieningKreuzer2008} so that, additionally, the error $\|\sqrt{\varrho}\nabla(u - v_C)\|^2_2$ with $\varrho \coloneqq (|\nabla u| + |\nabla v_C|)^{2}$ can be controlled.

On uniform meshes, \Cref{fig:plaplace_conv}(a) displays the convergence rates $2/3$ for $\eta$. Adaptive computation refines towards the re-entrant corner  in \Cref{fig:plaplace_conv}(b) recover the optimal convergence rates $\mathrm{ndof}^{-k}$ for all displayed polynomial degrees $k$. For $k = 1$, the empirical results are consistent with the known optimality of adaptive algorithms for P1 conforming discretizations in \cite{DieningKreuzer2008,BelenkiDieningKreuzer2012}.

The choice $r = 2$, $s = 1$ appears to work well for different values of $p$ as displayed in \Cref{fig:plaplace_comp}(a), where we empirically observe the optimal convergence rates $2$ for $\eta$ with the fixed polynomial degree $k = 2$. More sophisticated solvers \cite{Kacanov2020,Kacanov2023} are required to handle large $p$.
For different choices of $r$ and $s$, such as $r = p$ and $s = p-1$ from \cite{BurmanErn2008} or $r = p$ and $s = (k+1)(p-2)$ from \Cref{rem:choice_s}, \Cref{fig:plaplace_comp}(b) shows similar convergence behavior in adaptive computations with $k = 2$ and $p = 4$. We mention that the iterative solver in $\texttt{fminunc}$ requires significantly less Newton and cg iterations for $r = 2$ and $s = 1$ than the other choices.
A thorough study on solver performance is, however, not the focus and beyond the scope of this paper.

\begin{figure}[ht]
	\begin{minipage}[b]{0.475\textwidth}
		\centering
		\includegraphics[height=5cm]{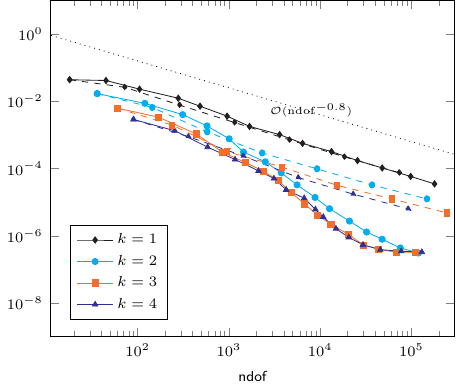}
	\end{minipage}\hfill
	\begin{minipage}[b]{0.475\textwidth}
		\centering
		\includegraphics[height=5cm]{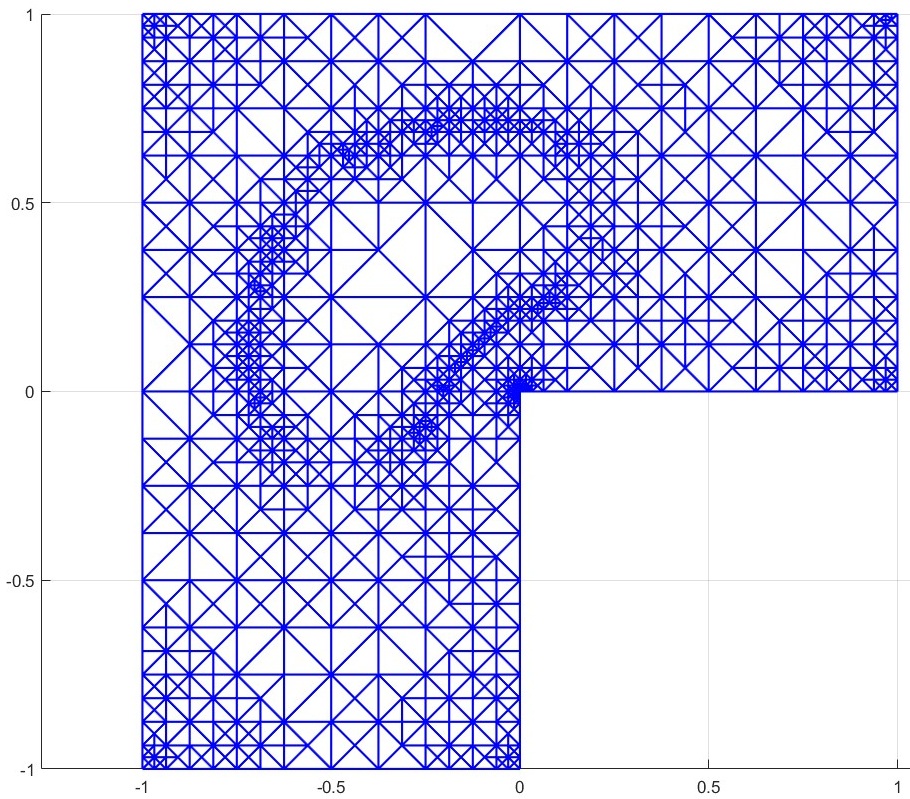}
	\end{minipage}
	\captionsetup{width=1\linewidth}
	\caption{(a) Convergence history plot of $\eta$ for $k = 1, \dots, 4$ and adaptive triangulation in the Bingham flow problem of \Cref{sec:bingham}. The results are obtained with $\varepsilon = 10^{-5}$}
	\label{fig:bingham}
\end{figure}
\begin{figure}[ht]
	\includegraphics[height=8cm]{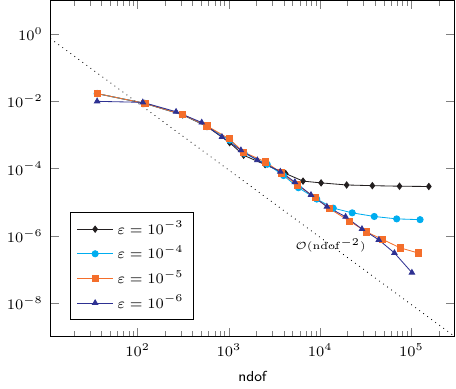}
	\captionsetup{width=0.8\linewidth}
	\caption{Convergence history plot of $\eta$ for $k = 2$ and $\varepsilon = 10^{-3}, \dots, 10^{-6}$ in \Cref{sec:bingham}}
	\label{fig:bingham-II}
\end{figure}

\subsection{Bingham flow through a pipe}\label{sec:bingham}
Given fixed positive parameters $\mu, g > 0$, the modelling of a uni-directional flow through a pipe with cross-section $\Omega \subset \mathbb{R}^2$ leads to the minimization problem \eqref{def:energy} with the energy density
\begin{align*}
	W(a) \coloneqq \mu|a|^2/2 + g|a| \quad\text{for any } a \in \mathbb{R}^2,
\end{align*} 
cf.~\cite{DuvantLions1972,CarstenReddySchedensack2016}, and $\Gamma_\mathrm{D} = \partial \Omega$.
An explicit computation \cite{Tran2024} shows that
\begin{align*}
	W^*(\alpha) = \begin{cases}
		0 &\mbox{if } |\alpha| \leq g\\
		(|\alpha| - g)^2/(2 \mu) &\mbox{if } |\alpha| > g.
	\end{cases}
\end{align*}
The strict convexity of $W$ leads to a unique minimizer $u$ of $E$ in $V$. Although $W$ is not differentiable, there exists $\sigma \in H(\div,\Omega) = W^2(\div,\Omega)$ such that $\sigma \in \partial W(\nabla u)$ and $\div\, \sigma = -f$ pointwise a.e.~in $\Omega$ \cite[Chapter II, Theorem 6.3]{Glowinski2008}.
Thus, there is no duality gap $E(u) = E^*(\sigma)$.
Furthermore, \eqref{ineq:a-posteriori} is satisfied with $\delta(\alpha,\beta) \coloneqq \|\alpha - \beta\|^2$ for any $\alpha, \beta \in L^2(\Omega)^2$ \cite[Lemma 1]{CarstenReddySchedensack2016}.

The postprocessings for \eqref{ineq:a-posteriori} are obtained from a regularized discrete problem as in \cite{CarstenReddySchedensack2016,Tran2024}. Given $\varepsilon > 0$, define $W_\varepsilon \in C^1(\mathbb{R}^2)$ by
\begin{align*}
	W_\varepsilon(a) \coloneqq \mu|a|^2/2 + g(\sqrt{|a|^2 + \varepsilon^2}) \quad\text{for any } a \in \mathbb{R}^2.
\end{align*}
The unique minimizer $u_{h,\varepsilon} \in V_h$ of the discrete energy
\begin{align*}
	E_{h,\varepsilon}(v_h) \coloneqq \int_\Omega (W_\varepsilon(\nabla_h v_h) - f v_h) \d{x} + \s_h(v_h)/2
\end{align*}
among $v_h \in V_h$
allows for the postprocessings $\sigma_\RT \in H(\div,\Omega)$ with $\div\,\sigma_\RT = -f$ from as in \Cref{thm:post-processing} and $v_C \in V_h \cap V$ as the nodal average of $u_{h,\varepsilon}$.

The computer experiment runs $\mu = 1$, $g = 0.2$, $f \equiv 1$, and 
shows the convergence rate 0.8  on uniform meshes in \Cref{fig:bingham}(a). 
Adaptive computation refines towards a parameter dependent region and towards 
the re-entrant corner. This leads to a significant improvement for higher-order discretizations $k \geq 2$. Empirical convergence rates are difficult to determine as a plateau is reached due to the regularization. \Cref{fig:bingham}(b) displays the effect of regularization on the error estimator for different 
parameters $\varepsilon$.

\subsection{Conclusions}
The numerical experiments of this section provide similar empirical results to those of \cite{Tran2024} with the hybridizable method outlined in \Cref{rem:LS-stab}. Adaptive mesh-refining leads to improved convergence rates for the a~posteriori error estimator compared to uniform mesh refinements. For the $p$-Laplace problem, optimal convergence rates are recovered with a quadratic stabilization by adaptive mesh-refining algorithms for singular solutions. In fact we
recommend $r=2$ for nonlinear problems. 

\printbibliography

@article {DiPietroDroniou2017,
    AUTHOR = {Di Pietro, Daniele A. and Droniou, J\'{e}r\^{o}me},
     TITLE = {A hybrid high-order method for {L}eray-{L}ions elliptic
              equations on general meshes},
   JOURNAL = {Math. Comp.},
  FJOURNAL = {Mathematics of Computation},
    VOLUME = {86},
      YEAR = {2017},
    NUMBER = {307},
     PAGES = {2159--2191},
      ISSN = {0025-5718},
   MRCLASS = {65N30 (65N08 65N12)},
  MRNUMBER = {3647954},
MRREVIEWER = {Constantin B\u{a}cu\c{t}\u{a}},
       DOI = {10.1090/mcom/3180},
       URL = {https://doi.org/10.1090/mcom/3180},
}

@article {CMueller02,
	AUTHOR = {Carstensen, Carsten and M\"uller, Stefan},
	TITLE = {Local stress regularity in scalar nonconvex variational
	problems},
	JOURNAL = {SIAM J. Math. Anal.},
	FJOURNAL = {SIAM Journal on Mathematical Analysis},
	VOLUME = {34},
	YEAR = {2002},
	NUMBER = {2},
	PAGES = {495--509},
	ISSN = {0036-1410},
	MRCLASS = {49N60 (35D10 35J70 74B15 74G65 74N15)},
	MRNUMBER = {1951785},
	MRREVIEWER = {Tom\'a\v s Roub\'\i \v cek},
	DOI = {10.1137/S0036141001396436},
	URL = {https://doi.org/10.1137/S0036141001396436},
}

@article {CPlechac1997,
	AUTHOR = {Carstensen, Carsten and Plech\'a\v{c}, Petr},
	TITLE = {Numerical solution of the scalar double-well problem allowing
	microstructure},
	JOURNAL = {Math. Comp.},
	FJOURNAL = {Mathematics of Computation},
	VOLUME = {66},
	YEAR = {1997},
	NUMBER = {219},
	PAGES = {997--1026},
	ISSN = {0025-5718},
	MRCLASS = {65N30 (73C50 73V05 73V25)},
	MRNUMBER = {1415798},
	MRREVIEWER = {Weimin Han},
	URL = {https://doi.org/10.1090/S0025-5718-97-00849-1},
}

@article {CLiu2015,
	AUTHOR = {Carstensen, C. and Liu, D. J.},
	TITLE = {Nonconforming {FEM}s for an optimal design problem},
	JOURNAL = {SIAM J. Numer. Anal.},
	FJOURNAL = {SIAM Journal on Numerical Analysis},
	VOLUME = {53},
	YEAR = {2015},
	NUMBER = {2},
	PAGES = {874--894},
	ISSN = {0036-1429},
	MRCLASS = {65N30 (65N12 65Y20 74P15)},
	MRNUMBER = {3327357},
	MRREVIEWER = {Gerhard Starke},
	URL = {https://doi.org/10.1137/130927103},
}

@article {Knees2008,
	AUTHOR = {Knees, Dorothee},
	TITLE = {Global stress regularity of convex and some nonconvex
	variational problems},
	JOURNAL = {Ann. Mat. Pura Appl. (4)},
	FJOURNAL = {Annali di Matematica Pura ed Applicata. Series IV},
	VOLUME = {187},
	YEAR = {2008},
	NUMBER = {1},
	PAGES = {157--184},
	ISSN = {0373-3114},
	MRCLASS = {49N60 (35B65 35J85 74B05 74G65)},
	MRNUMBER = {2346014},
	URL = {https://doi.org/10.1007/s10231-006-0039-5},
}

@article {KohnStrang1986,
    AUTHOR = {Kohn, Robert V. and Strang, Gilbert},
     TITLE = {Optimal design and relaxation of variational problems. {I}},
   JOURNAL = {Comm. Pure Appl. Math.},
  FJOURNAL = {Communications on Pure and Applied Mathematics},
    VOLUME = {39},
      YEAR = {1986},
    NUMBER = {1},
     PAGES = {113--137},
      ISSN = {0010-3640},
   MRCLASS = {49A34 (73K40)},
  MRNUMBER = {820342},
MRREVIEWER = {Bernard Dacorogna},
       DOI = {10.1002/cpa.3160390107},
       URL = {https://doi.org/10.1002/cpa.3160390107},
}

@article {BartelsC2008,
    AUTHOR = {Bartels, S\"{o}ren and Carstensen, Carsten},
     TITLE = {A convergent adaptive finite element method for an optimal
              design problem},
   JOURNAL = {Numer. Math.},
  FJOURNAL = {Numerische Mathematik},
    VOLUME = {108},
      YEAR = {2008},
    NUMBER = {3},
     PAGES = {359--385},
      ISSN = {0029-599X},
   MRCLASS = {65N30 (74S05)},
  MRNUMBER = {2365822},
MRREVIEWER = {Christoph Ortner},
       DOI = {10.1007/s00211-007-0122-x},
       URL = {https://doi.org/10.1007/s00211-007-0122-x},
}

@article {CJochimsen2003,
	AUTHOR = {Carstensen, C. and Jochimsen, K.},
	TITLE = {Adaptive finite element methods for microstructures?
	{N}umerical experiments for a 2-well benchmark},
	JOURNAL = {Computing},
	FJOURNAL = {Computing. Archives for Scientific Computing},
	VOLUME = {71},
	YEAR = {2003},
	NUMBER = {2},
	PAGES = {175--204},
	ISSN = {0010-485X},
	MRCLASS = {74S05 (65N30 65N50 74N15)},
	MRNUMBER = {2016254},
	MRREVIEWER = {Michael Griebel},
	DOI = {10.1007/s00607-003-0027-1},
	URL = {https://doi.org/10.1007/s00607-003-0027-1},
}

@book {Bartels2015,
	AUTHOR = {Bartels, S\"{o}ren},
	TITLE = {Numerical methods for nonlinear partial differential
	equations},
	SERIES = {Springer Series in Computational Mathematics},
	VOLUME = {47},
	PUBLISHER = {Springer, Cham},
	YEAR = {2015},
	PAGES = {x+393},
	ISBN = {978-3-319-13796-4; 978-3-319-13797-1},
	MRCLASS = {65-01 (35A15 35A35 65Mxx 65Nxx)},
	MRNUMBER = {3309171},
	MRREVIEWER = {Karsten Urban},
	DOI = {10.1007/978-3-319-13797-1},
	URL = {https://doi.org/10.1007/978-3-319-13797-1},
}

@article {BarrettLiu1994,
	AUTHOR = {Barrett, John W. and Liu, W. B.},
	TITLE = {Quasi-norm error bounds for the finite element approximation
	of a non-{N}ewtonian flow},
	JOURNAL = {Numer. Math.},
	FJOURNAL = {Numerische Mathematik},
	VOLUME = {68},
	YEAR = {1994},
	NUMBER = {4},
	PAGES = {437--456},
	ISSN = {0029-599X},
	MRCLASS = {65N30 (65N15 76A05 76M10)},
	MRNUMBER = {1301740},
	MRREVIEWER = {G. W. Hedstrom},
	DOI = {10.1007/s002110050071},
	URL = {https://doi.org/10.1007/s002110050071},
}

@article {BarrettLiu1993,
	AUTHOR = {Barrett, John W. and Liu, W. B.},
	TITLE = {Finite element approximation of the {$p$}-{L}aplacian},
	JOURNAL = {Math. Comp.},
	FJOURNAL = {Mathematics of Computation},
	VOLUME = {61},
	YEAR = {1993},
	NUMBER = {204},
	PAGES = {523--537},
	ISSN = {0025-5718},
	MRCLASS = {65N30},
	MRNUMBER = {1192966},
	MRREVIEWER = {Qian Li},
	DOI = {10.2307/2153239},
	URL = {https://doi.org/10.2307/2153239},
}

@article {AlbertyCFunken1999,
	AUTHOR = {Alberty, Jochen and Carstensen, Carsten and Funken, Stefan A.},
	TITLE = {Remarks around 50 lines of {M}atlab: short finite element
	implementation},
	JOURNAL = {Numer. Algorithms},
	FJOURNAL = {Numerical Algorithms},
	VOLUME = {20},
	YEAR = {1999},
	NUMBER = {2-3},
	PAGES = {117--137},
	ISSN = {1017-1398},
	MRCLASS = {65N30 (65M60)},
	MRNUMBER = {1709562},
	DOI = {10.1023/A:1019155918070},
	URL = {https://doi.org/10.1023/A:1019155918070},
}

@article {BuffaOrtner2009,
	AUTHOR = {Buffa, Annalisa and Ortner, Christoph},
	TITLE = {Compact embeddings of broken {S}obolev spaces and
	applications},
	JOURNAL = {IMA J. Numer. Anal.},
	FJOURNAL = {IMA Journal of Numerical Analysis},
	VOLUME = {29},
	YEAR = {2009},
	NUMBER = {4},
	PAGES = {827--855},
	ISSN = {0272-4979},
	MRCLASS = {65J05 (65N30)},
	MRNUMBER = {2557047},
	DOI = {10.1093/imanum/drn038},
	URL = {https://doi.org/10.1093/imanum/drn038},
}

@article {ErnZanotti2020,
	AUTHOR = {Ern, Alexandre and Zanotti, Pietro},
	TITLE = {A quasi-optimal variant of the hybrid high-order method for
	elliptic partial differential equations with {$H^{-1}$} loads},
	JOURNAL = {IMA J. Numer. Anal.},
	FJOURNAL = {IMA Journal of Numerical Analysis},
	VOLUME = {40},
	YEAR = {2020},
	NUMBER = {4},
	PAGES = {2163--2188},
	ISSN = {0272-4979},
	MRCLASS = {65N99 (65N15)},
	MRNUMBER = {4167044},
	MRREVIEWER = {Samuel Cogar},
	DOI = {10.1093/imanum/drz057},
	URL = {https://doi.org/10.1093/imanum/drz057},
}

@article {GlowinskiMarrocco1975,
	AUTHOR = {Glowinski, R. and Marrocco, A.},
	TITLE = {Sur l'approximation, par \'{e}l\'{e}ments finis d'ordre un, et la
	r\'{e}solution, par p\'{e}nalisation-dualit\'{e}, d'une classe de
	probl\`emes de {D}irichlet non lin\'{e}aires},
	JOURNAL = {RAIRO S\'{e}r. Rouge Anal. Num\'{e}r.},
	FJOURNAL = {Revue Fran\c{c}aise d'Automatique, Informatique et Recherche
	Op\'{e}rationnelle S\'{e}rie Rouge. Analyse Num\'{e}rique},
	VOLUME = {9},
	YEAR = {1975},
	NUMBER = {{\rm R}-2},
	PAGES = {41--76},
	ISSN = {0397-9342},
	MRCLASS = {65N35},
	MRNUMBER = {388811},
	MRREVIEWER = {J. R. Cannon},
}

@article {DiPietroDroniou2017-II,
	AUTHOR = {Di Pietro, Daniele A. and Droniou, J\'{e}r\^{o}me},
	TITLE = {{$W^{s,p}$}-approximation properties of elliptic projectors on
	polynomial spaces, with application to the error analysis of a
	hybrid high-order discretisation of {L}eray-{L}ions problems},
	JOURNAL = {Math. Models Methods Appl. Sci.},
	FJOURNAL = {Mathematical Models and Methods in Applied Sciences},
	VOLUME = {27},
	YEAR = {2017},
	NUMBER = {5},
	PAGES = {879--908},
	ISSN = {0218-2025},
	MRCLASS = {65N30 (65N08 65N12)},
	MRNUMBER = {3636615},
	MRREVIEWER = {V\'{\i}t Dolej\v{s}\'{\i}},
	DOI = {10.1142/S0218202517500191},
	URL = {https://doi.org/10.1142/S0218202517500191},
}

@book {AdamsFournier2003,
	AUTHOR = {Adams, Robert A. and Fournier, John J. F.},
	TITLE = {Sobolev spaces},
	SERIES = {Pure and Applied Mathematics (Amsterdam)},
	VOLUME = {140},
	EDITION = {Second},
	PUBLISHER = {Elsevier/Academic Press, Amsterdam},
	YEAR = {2003},
	PAGES = {xiv+305},
	ISBN = {0-12-044143-8},
	MRCLASS = {46E35 (46-01 46-02 46B70 46Exx)},
	MRNUMBER = {2424078},
}

@article {BelenkiDieningKreuzer2012,
	AUTHOR = {Belenki, Liudmila and Diening, Lars and Kreuzer, Christian},
	TITLE = {Optimality of an adaptive finite element method for the
	{$p$}-{L}aplacian equation},
	JOURNAL = {IMA J. Numer. Anal.},
	FJOURNAL = {IMA Journal of Numerical Analysis},
	VOLUME = {32},
	YEAR = {2012},
	NUMBER = {2},
	PAGES = {484--510},
	ISSN = {0272-4979},
	MRCLASS = {65N30 (65N15)},
	MRNUMBER = {2911397},
	MRREVIEWER = {Snorre H. Christiansen},
	DOI = {10.1093/imanum/drr016},
	URL = {https://doi.org/10.1093/imanum/drr016},
}

@article {CarstensenFeischlPagePraetorius2014,
	AUTHOR = {Carstensen, C. and Feischl, M. and Page, M. and Praetorius,
	D.},
	TITLE = {Axioms of adaptivity},
	JOURNAL = {Comput. Math. Appl.},
	FJOURNAL = {Computers \& Mathematics with Applications. An International
	Journal},
	VOLUME = {67},
	YEAR = {2014},
	NUMBER = {6},
	PAGES = {1195--1253},
	ISSN = {0898-1221},
	MRCLASS = {65N50 (65N12 65N22 65N30 65N38)},
	MRNUMBER = {3170325},
	MRREVIEWER = {Tsu-Fen Chen},
	DOI = {10.1016/j.camwa.2013.12.003},
	URL = {https://doi.org/10.1016/j.camwa.2013.12.003},
}

@article {DieningKreuzer2008,
	AUTHOR = {Diening, Lars and Kreuzer, Christian},
	TITLE = {Linear convergence of an adaptive finite element method for
	the {$p$}-{L}aplacian equation},
	JOURNAL = {SIAM J. Numer. Anal.},
	FJOURNAL = {SIAM Journal on Numerical Analysis},
	VOLUME = {46},
	YEAR = {2008},
	NUMBER = {2},
	PAGES = {614--638},
	ISSN = {0036-1429},
	MRCLASS = {65N30 (35J60 35J70)},
	MRNUMBER = {2383205},
	MRREVIEWER = {Olivier Besson},
	DOI = {10.1137/070681508},
	URL = {https://doi.org/10.1137/070681508},
}

@book {DiPietroErn2012,
	AUTHOR = {Di Pietro, Daniele Antonio and Ern, Alexandre},
	TITLE = {Mathematical aspects of discontinuous {G}alerkin methods},
	SERIES = {Math\'{e}matiques \& Applications (Berlin) [Mathematics \&
	Applications]},
	VOLUME = {69},
	PUBLISHER = {Springer, Heidelberg},
	YEAR = {2012},
	PAGES = {xviii+384},
	ISBN = {978-3-642-22979-4},
	MRCLASS = {65-02 (35A35 35F15 35J25 35Q35 65M60 65N30)},
	MRNUMBER = {2882148},
	MRREVIEWER = {R\'{e}mi Vaillancourt},
	DOI = {10.1007/978-3-642-22980-0},
	URL = {https://doi.org/10.1007/978-3-642-22980-0},
}

@article {BrezziManziniMariniPietraRusso2000,
	AUTHOR = {Brezzi, F. and Manzini, G. and Marini, D. and Pietra, P. and
	Russo, A.},
	TITLE = {Discontinuous {G}alerkin approximations for elliptic problems},
	JOURNAL = {Numer. Methods Partial Differential Equations},
	FJOURNAL = {Numerical Methods for Partial Differential Equations. An
	International Journal},
	VOLUME = {16},
	YEAR = {2000},
	NUMBER = {4},
	PAGES = {365--378},
	ISSN = {0749-159X,1098-2426},
	MRCLASS = {65N30 (65N12)},
	MRNUMBER = {1765651},
	MRREVIEWER = {K.\ Najzar},
	DOI = {10.1002/1098-2426(200007)16:4<365::AID-NUM2>3.0.CO;2-Y},
	URL =
	{https://doi.org/10.1002/1098-2426(200007)16:4<365::AID-NUM2>3.0.CO;2-Y},
}

@article {CockburnShu1998,
	AUTHOR = {Cockburn, Bernardo and Shu, Chi-Wang},
	TITLE = {The local discontinuous {G}alerkin method for time-dependent
	convection-diffusion systems},
	JOURNAL = {SIAM J. Numer. Anal.},
	FJOURNAL = {SIAM Journal on Numerical Analysis},
	VOLUME = {35},
	YEAR = {1998},
	NUMBER = {6},
	PAGES = {2440--2463},
	ISSN = {0036-1429,1095-7170},
	MRCLASS = {65M60 (76M25)},
	MRNUMBER = {1655854},
	MRREVIEWER = {Leonid\ K.\ Antanovski\u i},
	DOI = {10.1137/S0036142997316712},
	URL = {https://doi.org/10.1137/S0036142997316712},
}

@article {BurmanErn2008,
	AUTHOR = {Burman, Erik and Ern, Alexandre},
	TITLE = {Discontinuous {G}alerkin approximation with discrete
	variational principle for the nonlinear {L}aplacian},
	JOURNAL = {C. R. Math. Acad. Sci. Paris},
	FJOURNAL = {Comptes Rendus Math\'ematique. Acad\'emie des Sciences. Paris},
	VOLUME = {346},
	YEAR = {2008},
	NUMBER = {17-18},
	PAGES = {1013--1016},
	ISSN = {1631-073X,1778-3569},
	MRCLASS = {65N30},
	MRNUMBER = {2449647},
	DOI = {10.1016/j.crma.2008.07.005},
	URL = {https://doi.org/10.1016/j.crma.2008.07.005},
}

@article {Tran2024,
	AUTHOR = {Tran, Ngoc Tien},
	TITLE = {Discrete weak duality of hybrid high-order methods for convex
	minimization problems},
	JOURNAL = {SIAM J. Numer. Anal.},
	FJOURNAL = {SIAM Journal on Numerical Analysis},
	VOLUME = {62},
	YEAR = {2024},
	NUMBER = {4},
	PAGES = {1492--1514},
	ISSN = {0036-1429,1095-7170},
	MRCLASS = {65N12 (65N30 65Y20)},
	MRNUMBER = {4768465},
	MRREVIEWER = {Wenyi\ Tian},
	DOI = {10.1137/23M1594534},
	URL = {https://doi.org/10.1137/23M1594534},
}

@article {CarstensenTran2021,
	AUTHOR = {Carstensen, C. and Tran, N. T.},
	TITLE = {Unstabilized {H}ybrid {H}igh-order {M}ethod for a {C}lass of
	{D}egenerate {C}onvex {M}inimization {P}roblems},
	JOURNAL = {SIAM J. Numer. Anal.},
	FJOURNAL = {SIAM Journal on Numerical Analysis},
	VOLUME = {59},
	YEAR = {2021},
	NUMBER = {3},
	PAGES = {1348--1373},
}

@article {BartelsKaltenbach2023,
    AUTHOR = {Bartels, S. and Kaltenbach, A.},
     TITLE = {Explicit and efficient error estimation for convex
              minimization problems},
   JOURNAL = {Math. Comp.},
  FJOURNAL = {Mathematics of Computation},
    VOLUME = {92},
      YEAR = {2023},
    NUMBER = {343},
     PAGES = {2247--2279},
}

@article {Doerfler1996,
    AUTHOR = {D\"orfler, Willy},
     TITLE = {A convergent adaptive algorithm for {P}oisson's equation},
   JOURNAL = {SIAM J. Numer. Anal.},
  FJOURNAL = {SIAM Journal on Numerical Analysis},
    VOLUME = {33},
      YEAR = {1996},
    NUMBER = {3},
     PAGES = {1106--1124},
      ISSN = {0036-1429},
   MRCLASS = {65N50 (65N55)},
  MRNUMBER = {1393904},
MRREVIEWER = {S.\ F.\ McCormick},
       DOI = {10.1137/0733054},
       URL = {https://doi.org/10.1137/0733054},
}

@article {CarstenReddySchedensack2016,
    AUTHOR = {Carstensen, C. and Reddy, B. D. and Schedensack, M.},
     TITLE = {A natural nonconforming {FEM} for the {B}ingham flow problem
              is quasi-optimal},
   JOURNAL = {Numer. Math.},
  FJOURNAL = {Numerische Mathematik},
    VOLUME = {133},
      YEAR = {2016},
    NUMBER = {1},
     PAGES = {37--66},
}

@book {Glowinski2008,
    AUTHOR = {Glowinski, R.},
     TITLE = {Numerical methods for nonlinear variational problems},
    SERIES = {Scientific Computation},
      NOTE = {Reprint of the 1984 original},
 PUBLISHER = {Springer-Verlag, Berlin},
      YEAR = {2008},
     PAGES = {xviii+493},
}

@book {DuvantLions1972,
    AUTHOR = {Duvaut, G. and Lions, J.-L.},
     TITLE = {Les in\'equations en m\'ecanique et en physique},
    SERIES = {Travaux et Recherches Math\'ematiques},
    VOLUME = {No. 21},
 PUBLISHER = {Dunod, Paris},
      YEAR = {1972},
     PAGES = {xx+387},
   MRCLASS = {73.49},
  MRNUMBER = {464857},
}

@article {Chow1989,
    AUTHOR = {Chow, S.-S.},
     TITLE = {Finite element error estimates for nonlinear elliptic
              equations of monotone type},
   JOURNAL = {Numer. Math.},
  FJOURNAL = {Numerische Mathematik},
    VOLUME = {54},
      YEAR = {1989},
    NUMBER = {4},
     PAGES = {373--393},
      ISSN = {0029-599X,0945-3245},
   MRCLASS = {65N30},
  MRNUMBER = {972416},
MRREVIEWER = {Ji\v r\'i\ Nedoma},
       DOI = {10.1007/BF01396320},
       URL = {https://doi.org/10.1007/BF01396320},
}

@book {NeittaanmakiRepin2004,
    AUTHOR = {Neittaanm\"aki, P. and Repin, S.},
     TITLE = {Reliable methods for computer simulation},
    SERIES = {Studies in Mathematics and its Applications},
    VOLUME = {33},
      NOTE = {Error control and a posteriori estimates},
 PUBLISHER = {Elsevier Science B.V., Amsterdam},
      YEAR = {2004},
     PAGES = {x+305},
      ISBN = {0-444-51376-0},
   MRCLASS = {65-02 (65N15)},
  MRNUMBER = {2095603},
MRREVIEWER = {\'Oscar\ L\'opez Pouso},
}

@article {Repin1997,
    AUTHOR = {Repin, S. I.},
     TITLE = {A posteriori error estimation for nonlinear variational
              problems by duality theory},
   JOURNAL = {Zap. Nauchn. Sem. S.-Peterburg. Otdel. Mat. Inst. Steklov.
              (POMI)},
  FJOURNAL = {Rossi\u iskaya Akademiya Nauk. Sankt-Peterburgskoe Otdelenie.
              Matematicheski\u i\ Institut im. V. A. Steklova. Zapiski
              Nauchnykh Seminarov (POMI)},
    VOLUME = {243},
      YEAR = {1997},
     PAGES = {201--214, 342},
      ISSN = {0373-2703},
   MRCLASS = {49N15 (65J15)},
  MRNUMBER = {1629741},
MRREVIEWER = {Rolf\ Kl\"otzler},
       DOI = {10.1007/BF02673600},
       URL = {https://doi.org/10.1007/BF02673600},
}

@article {Repin2000,
    AUTHOR = {Repin, Sergey I.},
     TITLE = {A posteriori error estimation for variational problems with
              uniformly convex functionals},
   JOURNAL = {Math. Comp.},
  FJOURNAL = {Mathematics of Computation},
    VOLUME = {69},
      YEAR = {2000},
    NUMBER = {230},
     PAGES = {481--500},
      ISSN = {0025-5718,1088-6842},
   MRCLASS = {49M29 (65J15 65K10)},
  MRNUMBER = {1681096},
MRREVIEWER = {Viorel\ Arnautu},
       DOI = {10.1090/S0025-5718-99-01190-4},
       URL = {https://doi.org/10.1090/S0025-5718-99-01190-4},
}

@article {LuceWohlmuth2004,
    AUTHOR = {Luce, R. and Wohlmuth, B. I.},
     TITLE = {A local a posteriori error estimator based on equilibrated
              fluxes},
   JOURNAL = {SIAM J. Numer. Anal.},
  FJOURNAL = {SIAM Journal on Numerical Analysis},
    VOLUME = {42},
      YEAR = {2004},
    NUMBER = {4},
     PAGES = {1394--1414},
      ISSN = {0036-1429,1095-7170},
   MRCLASS = {65N15 (65N30 65N50)},
  MRNUMBER = {2114283},
MRREVIEWER = {S\"oren\ Bartels},
       DOI = {10.1137/S0036142903433790},
       URL = {https://doi.org/10.1137/S0036142903433790},
}

@article {BraessSchoeberl2008,
    AUTHOR = {Braess, Dietrich and Sch\"oberl, Joachim},
     TITLE = {Equilibrated residual error estimator for edge elements},
   JOURNAL = {Math. Comp.},
  FJOURNAL = {Mathematics of Computation},
    VOLUME = {77},
      YEAR = {2008},
    NUMBER = {262},
     PAGES = {651--672},
      ISSN = {0025-5718,1088-6842},
   MRCLASS = {65N30 (78M10)},
  MRNUMBER = {2373174},
MRREVIEWER = {Dalibor\ Luk\'a\v s},
       DOI = {10.1090/S0025-5718-07-02080-7},
       URL = {https://doi.org/10.1090/S0025-5718-07-02080-7},
}

@article {ErnVohralik2015,
    AUTHOR = {Ern, Alexandre and Vohral\'ik, Martin},
     TITLE = {Polynomial-degree-robust a posteriori estimates in a unified
              setting for conforming, nonconforming, discontinuous
              {G}alerkin, and mixed discretizations},
   JOURNAL = {SIAM J. Numer. Anal.},
  FJOURNAL = {SIAM Journal on Numerical Analysis},
    VOLUME = {53},
      YEAR = {2015},
    NUMBER = {2},
     PAGES = {1058--1081},
      ISSN = {0036-1429,1095-7170},
   MRCLASS = {65N30 (65N15)},
  MRNUMBER = {3335498},
MRREVIEWER = {Dominic\ Breit},
       DOI = {10.1137/130950100},
       URL = {https://doi.org/10.1137/130950100},
}

@article {CockburnGopalakrishnanLazarov2009,
    AUTHOR = {Cockburn, Bernardo and Gopalakrishnan, Jayadeep and Lazarov,
              Raytcho},
     TITLE = {Unified hybridization of discontinuous {G}alerkin, mixed, and
              continuous {G}alerkin methods for second order elliptic
              problems},
   JOURNAL = {SIAM J. Numer. Anal.},
  FJOURNAL = {SIAM Journal on Numerical Analysis},
    VOLUME = {47},
      YEAR = {2009},
    NUMBER = {2},
     PAGES = {1319--1365},
      ISSN = {0036-1429,1095-7170},
   MRCLASS = {65N30},
  MRNUMBER = {2485455},
MRREVIEWER = {Jose\ Luis\ Gracia},
       DOI = {10.1137/070706616},
       URL = {https://doi.org/10.1137/070706616},
}

@article {OrtnerSueli2007,
    AUTHOR = {Ortner, Christoph and S\"uli, Endre},
     TITLE = {Discontinuous {G}alerkin finite element approximation of
              nonlinear second-order elliptic and hyperbolic systems},
   JOURNAL = {SIAM J. Numer. Anal.},
  FJOURNAL = {SIAM Journal on Numerical Analysis},
    VOLUME = {45},
      YEAR = {2007},
    NUMBER = {4},
     PAGES = {1370--1397},
      ISSN = {0036-1429,1095-7170},
   MRCLASS = {74S05 (65M60 74H15)},
  MRNUMBER = {2338392},
MRREVIEWER = {Jos\'e\ R.\ Fern\'andez},
       DOI = {10.1137/06067119X},
       URL = {https://doi.org/10.1137/06067119X},
}

@article {GrekasKoumatosMakridakisVikelis2025,
    AUTHOR = {Grekas, G. and Koumatos, K. and Makridakis, C. and Vikelis,
              A.},
     TITLE = {A class of discontinuous {G}alerkin methods for nonlinear
              variational problems},
   JOURNAL = {Math. Comp.},
  FJOURNAL = {Mathematics of Computation},
    VOLUME = {94},
      YEAR = {2025},
    NUMBER = {355},
     PAGES = {2221--2250},
      ISSN = {0025-5718,1088-6842},
   MRCLASS = {65N30 (49J10 49J45 49M25)},
  MRNUMBER = {4919560},
MRREVIEWER = {Huipo\ Liu},
       DOI = {10.1090/mcom/4040},
       URL = {https://doi.org/10.1090/mcom/4040},
}

@article {EyckLew2006,
    AUTHOR = {Eyck, A. Ten and Lew, A.},
     TITLE = {Discontinuous {G}alerkin methods for non-linear elasticity},
   JOURNAL = {Internat. J. Numer. Methods Engrg.},
  FJOURNAL = {International Journal for Numerical Methods in Engineering},
    VOLUME = {67},
      YEAR = {2006},
    NUMBER = {9},
     PAGES = {1204--1243},
      ISSN = {0029-5981,1097-0207},
   MRCLASS = {74G15 (65N30 74B20 74S05)},
  MRNUMBER = {2248082},
       DOI = {10.1002/nme.1667},
       URL = {https://doi.org/10.1002/nme.1667},
}

@article {LDG2014,
    AUTHOR = {Diening, Lars and K\"oner, Dietmar and R{$\mathring{\mathrm{u}}$}{\v{z}}i{\v{c}}ka,
              Michael and Toulopoulos, Ioannis},
     TITLE = {A local discontinuous {G}alerkin approximation for systems
              with {$p$}-structure},
   JOURNAL = {IMA J. Numer. Anal.},
  FJOURNAL = {IMA Journal of Numerical Analysis},
    VOLUME = {34},
      YEAR = {2014},
    NUMBER = {4},
     PAGES = {1447--1488},
      ISSN = {0272-4979,1464-3642},
   MRCLASS = {65N30 (65N15)},
  MRNUMBER = {3269432},
       DOI = {10.1093/imanum/drt040},
       URL = {https://doi.org/10.1093/imanum/drt040},
}

@article {EbmeyerLiuSteinhauer2005,
    AUTHOR = {Ebmeyer, C. and Liu, W. B. and Steinhauer, M.},
     TITLE = {Global regularity in fractional order {S}obolev spaces for the
              {$p$}-{L}aplace equation on polyhedral domains},
   JOURNAL = {Z. Anal. Anwendungen},
  FJOURNAL = {Zeitschrift f\"ur Analysis und ihre Anwendungen. Journal for
              Analysis and its Applications},
    VOLUME = {24},
      YEAR = {2005},
    NUMBER = {2},
     PAGES = {353--374},
      ISSN = {0232-2064,1661-4534},
   MRCLASS = {35J60 (35B65 35D10 35J25)},
  MRNUMBER = {2174028},
MRREVIEWER = {Eugen\ Viszus},
       DOI = {10.4171/ZAA/1245},
       URL = {https://doi.org/10.4171/ZAA/1245},
}

@book {DroniouEymardGallouet2018,
    AUTHOR = {Droniou, J\'er\^ome and Eymard, Robert and Gallou\"et, Thierry
              and Guichard, Cindy and Herbin, Rapha\`ele},
     TITLE = {The gradient discretisation method},
    SERIES = {Math\'ematiques \& Applications (Berlin) [Mathematics \&
              Applications]},
    VOLUME = {82},
 PUBLISHER = {Springer, Cham},
      YEAR = {2018},
     PAGES = {xxiv+497},
      ISBN = {978-3-319-79042-8; 978-3-319-79041-1},
   MRCLASS = {65-02 (65Mxx 65Nxx)},
  MRNUMBER = {3898702},
       DOI = {10.1007/978-3-319-79042-8},
       URL = {https://doi.org/10.1007/978-3-319-79042-8},
}

@article {ErnStephansenVohralik2010,
    AUTHOR = {Ern, Alexandre and Stephansen, Annette F. and Vohral\'ik,
              Martin},
     TITLE = {Guaranteed and robust discontinuous {G}alerkin a posteriori
              error estimates for convection-diffusion-reaction problems},
   JOURNAL = {J. Comput. Appl. Math.},
  FJOURNAL = {Journal of Computational and Applied Mathematics},
    VOLUME = {234},
      YEAR = {2010},
    NUMBER = {1},
     PAGES = {114--130},
      ISSN = {0377-0427,1879-1778},
   MRCLASS = {65N30 (65N15 76S05)},
  MRNUMBER = {2601287},
MRREVIEWER = {Christian\ Vergara},
       DOI = {10.1016/j.cam.2009.12.009},
       URL = {https://doi.org/10.1016/j.cam.2009.12.009},
}

@book {EkelandTeman1999,
    AUTHOR = {Ekeland, Ivar and T\'emam, Roger},
     TITLE = {Convex analysis and variational problems},
    SERIES = {Classics in Applied Mathematics},
    VOLUME = {28},
   EDITION = {English},
      NOTE = {Translated from the French},
 PUBLISHER = {Society for Industrial and Applied Mathematics (SIAM),
              Philadelphia, PA},
      YEAR = {1999},
     PAGES = {xiv+402},
      ISBN = {0-89871-450-8},
   MRCLASS = {49-02 (01A75 49J53 90C46)},
  MRNUMBER = {1727362},
       DOI = {10.1137/1.9781611971088},
       URL = {https://doi.org/10.1137/1.9781611971088},
}

@article {Bartels2021,
    AUTHOR = {Bartels, S\"oren},
     TITLE = {Error estimates for a class of discontinuous {G}alerkin
              methods for nonsmooth problems via convex duality relations},
   JOURNAL = {Math. Comp.},
  FJOURNAL = {Mathematics of Computation},
    VOLUME = {90},
      YEAR = {2021},
    NUMBER = {332},
     PAGES = {2579--2602},
      ISSN = {0025-5718,1088-6842},
   MRCLASS = {65N30 (65N12 65N15)},
  MRNUMBER = {4305362},
MRREVIEWER = {Chang-Ock\ Lee},
       DOI = {10.1090/mcom/3656},
       URL = {https://doi.org/10.1090/mcom/3656},
}

@article{CarstensenMueller2002,
	TITLE = {Local stress regularity in scalar nonconvex variational problems},
	AUTHOR = {Carstensen, C. and M{\"u}ller, S.},
	JOURNAL = {SIAM J. Math. Anal.},
	FJOURNAL = {SIAM Journal on Mathematical Analysis},
	YEAR = {2002},
	VOLUME = {34},
	NUMBER = {2},
	PAGES = {495--509},
}

@article {VeeserZanotti2018,
    AUTHOR = {Veeser, Andreas and Zanotti, Pietro},
     TITLE = {Quasi-optimal nonconforming methods for symmetric elliptic
              problems. {III}---{D}iscontinuous {G}alerkin and other
              interior penalty methods},
   JOURNAL = {SIAM J. Numer. Anal.},
  FJOURNAL = {SIAM Journal on Numerical Analysis},
    VOLUME = {56},
      YEAR = {2018},
    NUMBER = {5},
     PAGES = {2871--2894},
      ISSN = {0036-1429,1095-7170},
   MRCLASS = {65N30 (65N12 65N15)},
  MRNUMBER = {3857891},
MRREVIEWER = {Francesco\ Zirilli},
       DOI = {10.1137/17M1151675},
       URL = {https://doi.org/10.1137/17M1151675},
}

@article {ErnVohralik2009,
    AUTHOR = {Ern, Alexandre and Vohral\'ik, Martin},
     TITLE = {Flux reconstruction and a posteriori error estimation for
              discontinuous {G}alerkin methods on general nonmatching grids},
   JOURNAL = {C. R. Math. Acad. Sci. Paris},
  FJOURNAL = {Comptes Rendus Math\'ematique. Acad\'emie des Sciences. Paris},
    VOLUME = {347},
      YEAR = {2009},
    NUMBER = {7-8},
     PAGES = {441--444},
      ISSN = {1631-073X,1778-3569},
   MRCLASS = {65N15 (65N30)},
  MRNUMBER = {2537245},
       DOI = {10.1016/j.crma.2009.01.017},
       URL = {https://doi.org/10.1016/j.crma.2009.01.017},
}

@article {Kacanov2020,
    AUTHOR = {Diening, L. and Fornasier, M. and Tomasi, R. and Wank, M.},
     TITLE = {A relaxed {K}a\v canov iteration for the {$p$}-{P}oisson
              problem},
   JOURNAL = {Numer. Math.},
  FJOURNAL = {Numerische Mathematik},
    VOLUME = {145},
      YEAR = {2020},
    NUMBER = {1},
     PAGES = {1--34},
      ISSN = {0029-599X,0945-3245},
   MRCLASS = {65N30 (35J70 35J92 65N12 65N22)},
  MRNUMBER = {4091593},
MRREVIEWER = {Michael\ Neilan},
       DOI = {10.1007/s00211-020-01107-1},
       URL = {https://doi.org/10.1007/s00211-020-01107-1},
}

@article {Kacanov2023,
    AUTHOR = {Balci, Anna Kh. and Diening, Lars and Storn, Johannes},
     TITLE = {Relaxed {K}a\v canov scheme for the {$p$}-{L}aplacian with
              large exponent},
   JOURNAL = {SIAM J. Numer. Anal.},
  FJOURNAL = {SIAM Journal on Numerical Analysis},
    VOLUME = {61},
      YEAR = {2023},
    NUMBER = {6},
     PAGES = {2775--2794},
      ISSN = {0036-1429,1095-7170},
   MRCLASS = {65N30 (35J70 35J92 65N22)},
  MRNUMBER = {4668389},
MRREVIEWER = {D.\ J.\ Liu},
       DOI = {10.1137/22M1528550},
       URL = {https://doi.org/10.1137/22M1528550},
}

\appendix

\section{A~priori error analysis of conforming methods}\label{appendix}
Energy error of conforming methods for the convex minimization \eqref{def:energy}
are certainly understood from the arguments of \cite{GlowinskiMarrocco1975,Chow1989,CPlechac1997}, but precise statements are 
rare and provided for completeness in this appendix to explain  \eqref{confenergyerror}.

Let $V_h \subset V$ be a conforming subspace of $V$
and throughout assume (B1) and $\psi(x,\bullet) \in C^1(\mathbb{R})$ 
for a.e.~$x \in \Omega$.
Given a discrete minimizer $u_h \in \arg\min E(V_h)$ of $E$ in $V_h$, let $\sigma_h \coloneqq \D W(\nabla u_h)$ denote the discrete stress. The Euler-Lagrange equations read
\begin{align}
    \int_\Omega \sigma \cdot \nabla v \d{x} &= - \int_\Omega \partial_u \psi(\bullet, u) v \d{x} \quad\text{for any } v \in V,\label{eq:ELE}\\
    \int_\Omega \sigma_h \cdot \nabla v_h \d{x} &= - \int_\Omega \partial_u \psi(\bullet, u_h) v_h \d{x} \quad\text{for any } v_h \in V_h.\label{eq:dELE-conf}
\end{align}

\begin{theorem}[a~priori of conforming methods]\label{thm:a-priori-conf}
Any $v_h \in V_h$ satisfies
    \begin{align*}
        E(u_h) - E(u) \leq \int_\Omega \big((\sigma - \sigma_h) \cdot \nabla (u - v_h) + (\partial_u \psi(\bullet,u) - \partial_u \psi(\bullet,u_h)) (u - v_h)\big) \d{x}.
\end{align*}
\end{theorem}

\begin{proof}
    The convexity of 
    $W$ and $\psi$ imply
    $0 \leq W(\nabla u) - W(\nabla u_h) - \sigma_h \cdot \nabla(u - u_h)$ 
    and $0 \leq \psi(\bullet, u) - \psi(\bullet, u_h) - \partial_u \psi(\bullet, u_h)(u - u_h)$ a.e.~in $\Omega$. An integration provides
\begin{align}\label{ineq:pr-a-priori-conf}
        E(u_h) - E(u) &\leq
        - \int_\Omega (\sigma_h \cdot \nabla (u - u_h) + \partial_u \psi(\bullet, u_h)(u - u_h)) \d{x}.
    \end{align}
Given $v_h \in V_h$, \eqref{eq:dELE-conf} proves $(\sigma_h, v_h - u_h)_{L^2(\Omega)} = (\partial_u \psi(\bullet, u_h), u_h - v_h)_{L^2(\Omega)}$. 
This, \eqref{ineq:pr-a-priori-conf}, and 
$(\sigma, \nabla(u - v_h))_{L^2(\Omega)} 
= -(\partial_u \psi(\bullet, u), u - v_h)_{L^2(\Omega)}$ reveal
\begin{align*}
E(u_h) - E(u) &\leq - \int_\Omega (\sigma_h \cdot \nabla (u - v_h) 
+ \partial_u \psi(\bullet, u_h) (u - v_h)) \d{x}\\
&\leq \int_\Omega \big((\sigma - \sigma_h) \cdot \nabla (u - v_h) 
+ (\partial_u \psi(\bullet,u) 
- \partial_u \psi(\bullet,u_h)) (u - v_h)\big) \d{x}.
\qedhere
\end{align*}
\end{proof}

The following estimates imply \eqref{confenergyerror}.

\begin{corollary}[convergence rates of conforming methods]
Suppose (B2) and (B4). Then
    \begin{align*}
        0 \leq E(u_h) - E(u) \lesssim \min_{v_h \in V_h} \|\nabla(u - v_h)\|_p.
    \end{align*}
Suppose (B2)--(B4). Then
    \begin{align*}
        0 \leq E(u_h) - E(u) \lesssim \min_{v_h \in V_h} \|\nabla(u - v_h)\|_p^2.
    \end{align*}
\end{corollary}

\begin{proof}
    For the linear right-hand side in (B4), 
    $\partial_u \psi(\bullet,u) - \partial_u \psi(\bullet,u_h) = 0$ 
    and \Cref{thm:a-priori-conf} imply
    \begin{align}\label{ineq:a-priori-conf-lin-rhs}
        E(u_h) - E(u) \leq
        \int_\Omega (\sigma - \sigma_h) \cdot \nabla (u - v_h) \d{x} \quad\text{for any } v_h \in V_h.
    \end{align}
    The two sided growth (B2) leads to a uniform bound for $\nabla u$, $\nabla u_h$ in $L^p(\Omega)^n$ and $\sigma$, $\sigma_h$ in $L^{p'}(\Omega)^n$ \cite{GlowinskiMarrocco1975, CPlechac1997}. Hence, the first assertion follows from \eqref{ineq:a-priori-conf-lin-rhs} and the H\"older inequality.
    If (B3) holds, then the choice $\alpha \coloneqq \nabla u$ and $\beta \coloneqq \nabla u_h$ in (B3) and \eqref{eq:ELE} reveal
    \begin{align*}
        \|\sigma - \sigma_h\|^2_{L^{p'}(\Omega)} \lesssim \int_{\Omega} (W(\nabla u_h) - W(\nabla u) - \sigma \cdot (u_h - u)) \d{x} = E(u_h) - E(u).
    \end{align*}
    This, \eqref{ineq:a-priori-conf-lin-rhs}, and a Young inequality conclude the second assertion.
\end{proof}
\end{document}